\documentclass[a4paper, english]{amsart}

\usepackage{amssymb}
\usepackage{babel}
\usepackage{enumitem}
\usepackage{hyperref}
\usepackage[utf8]{inputenc}
\usepackage{newunicodechar}
\usepackage{varioref}
\usepackage[arrow,curve,matrix]{xy}

\usepackage{colortbl}
\usepackage{graphicx}
\usepackage{tikz}

\usepackage{mathrsfs}

\definecolor{linkred}{rgb}{0.7,0.2,0.2}
\definecolor{linkblue}{rgb}{0,0.2,0.6}

\numberwithin{figure}{section}

\usepackage[hyperpageref]{backref}

\setdescription{labelindent=\parindent, leftmargin=2\parindent}
\setitemize[1]{labelindent=\parindent, leftmargin=2\parindent}
\setenumerate[1]{labelindent=0cm, leftmargin=*, widest=iiii}

\newunicodechar{α}{\ensuremath{\alpha}}
\newunicodechar{β}{\ensuremath{\beta}}
\newunicodechar{χ}{\ensuremath{\chi}}
\newunicodechar{δ}{\ensuremath{\delta}}
\newunicodechar{ε}{\ensuremath{\varepsilon}}
\newunicodechar{Δ}{\ensuremath{\Delta}}
\newunicodechar{η}{\ensuremath{\eta}}
\newunicodechar{γ}{\ensuremath{\gamma}}
\newunicodechar{Γ}{\ensuremath{\Gamma}}
\newunicodechar{ι}{\ensuremath{\iota}}
\newunicodechar{κ}{\ensuremath{\kappa}}
\newunicodechar{λ}{\ensuremath{\lambda}}
\newunicodechar{Λ}{\ensuremath{\Lambda}}
\newunicodechar{μ}{\ensuremath{\mu}}
\newunicodechar{ω}{\ensuremath{\omega}}
\newunicodechar{Ω}{\ensuremath{\Omega}}
\newunicodechar{π}{\ensuremath{\pi}}
\newunicodechar{φ}{\ensuremath{\phi}}
\newunicodechar{Φ}{\ensuremath{\Phi}}
\newunicodechar{ψ}{\ensuremath{\psi}}
\newunicodechar{Ψ}{\ensuremath{\Psi}}
\newunicodechar{ρ}{\ensuremath{\rho}}
\newunicodechar{σ}{\ensuremath{\sigma}}
\newunicodechar{Σ}{\ensuremath{\Sigma}}
\newunicodechar{τ}{\ensuremath{\tau}}
\newunicodechar{θ}{\ensuremath{\theta}}
\newunicodechar{Θ}{\ensuremath{\Theta}}

\newunicodechar{→}{\ensuremath{\to}}
\newunicodechar{⨯}{\ensuremath{\times}}
\newunicodechar{∪}{\ensuremath{\cup}}
\newunicodechar{∩}{\ensuremath{\cap}}
\newunicodechar{⊇}{\ensuremath{\supseteq}}
\newunicodechar{⊃}{\ensuremath{\supset}}
\newunicodechar{⊆}{\ensuremath{\subseteq}}
\newunicodechar{⊂}{\ensuremath{\subset}}
\newunicodechar{≥}{\ensuremath{\geq}}
\newunicodechar{≤}{\ensuremath{\leq}}
\newunicodechar{∈}{\ensuremath{\in}}
\newunicodechar{◦}{\ensuremath{\circ}}
\newunicodechar{°}{\ensuremath{^\circ}}
\newunicodechar{…}{\ifmmode\mathellipsis\else\textellipsis\fi}

\newunicodechar{¹}{\ensuremath{^1}}
\newunicodechar{²}{\ensuremath{^2}}
\newunicodechar{³}{\ensuremath{^3}}

\DeclareFontFamily{OMS}{rsfs}{\skewchar\font'60}
\DeclareFontShape{OMS}{rsfs}{m}{n}{<-5>rsfs5 <5-7>rsfs7 <7->rsfs10 }{}
\DeclareSymbolFont{rsfs}{OMS}{rsfs}{m}{n}
\DeclareSymbolFontAlphabet{\scr}{rsfs}
\DeclareSymbolFontAlphabet{\scr}{rsfs}

\DeclareMathOperator{\Hom}{Hom}

\DeclareMathOperator{\red}{red}
\DeclareMathOperator{\reg}{reg}

\DeclareMathOperator{\Spec}{Spec}

\DeclareMathOperator{\supp}{supp}
\DeclareMathOperator{\tor}{tor}

\newcommand{\sF}{\scr{F}}

\newcommand{\sI}{\scr{I}}

\newcommand{\sO}{\scr{O}}

\newcommand{\bA}{\mathbb{A}}

\newcommand{\bC}{\mathbb{C}}

\newcommand{\bH}{\mathbb{H}}

\newcommand{\bN}{\mathbb{N}}

\newcommand{\bZ}{\mathbb{Z}}

\theoremstyle{plain}   
\newtheorem{thm}{Theorem}[section]

\newtheorem{prop}[thm]{Proposition}

\theoremstyle{remark}

\newtheorem{c-n-d}[thm]{Claim and Definition}

\newtheorem*{rem-nonumber}{Remark}

\numberwithin{equation}{thm}

\setlist[enumerate]{label=(\thethm.\arabic*), before={\setcounter{enumi}{\value{equation}}}, after={\setcounter{equation}{\value{enumi}}}}

\newcommand{\wtilde}{\widetilde}

\def\clap#1{\hbox to 0pt{\hss#1\hss}}

\def\mathclap{\mathpalette\mathclapinternal}

\def\mathclapinternal#1#2{%
\clap{$\mathsurround=0pt#1{#2}$}}

\newcommand{\factor}[2]{\left. \raise 2pt\hbox{$#1$} \right/\hskip -2pt\raise -2pt\hbox{$#2$}}%
\newcommand{\Preprint}[1]{}

\newcommand{\subversionInfo}{}
\newcommand{\svnid}[1]{}
\newcommand{\approvals}[1]{}

\newcommand{\Zar}{\mathrm{Zar}}
\newcommand{\cdh}{\mathrm{cdh}}
\newcommand{\cdp}{\mathrm{cdp}}
\newcommand{\sdh}{\mathrm{sdh}}
\newcommand{\h}{\mathrm{h}}
\newcommand{\eh}{\mathrm{eh}}
\newcommand{\rh}{\mathrm{rh}}

\newcommand{\salt}{\mathrm{s\mbox{-}alt}}
\newcommand{\stimes}{⨯^{\salt}}
\newcommand{\isom}{\cong}
\newcommand{\tensor}{\otimes}
\newcommand{\Frac}{\mathsf{Frac}} 
\newcommand{\RZ}{\mathsf{RZ}} 

\newcommand{\Sch}{\mathsf{Sch}}
\newcommand{\Sm}{\mathsf{Sm}}
\newcommand{\Reg}{\mathsf{Reg}}
\newcommand{\Cur}{\mathsf{Dvr}}
\newcommand{\ess}{\ensuremath{^{\mathsf{ess}}}}
\newcommand{\Pro}{\mathsf{Pro}}

\newcommand{\rs}{\mathrm{dvr}}
\newcommand{\Mor}{\mathrm{Mor}}

\theoremstyle{theorem}
\newtheorem{theo}[thm]{Theorem}
\newtheorem{coro}[thm]{Corollary}
\newtheorem{lemm}[thm]{Lemma}
\newtheorem{propo}[thm]{Proposition}

\theoremstyle{definition}
\newtheorem{defi}[thm]{Definition}
\newtheorem{propdef}[thm]{Definition and Proposition}
\newtheorem{obse}[thm]{Observation}
\newtheorem{rema}[thm]{Remark}
\newtheorem{remi}[thm]{Reminder}
\newtheorem{exam}[thm]{Example}
\newtheorem{summ}[thm]{Summary}
\newtheorem{nota}[thm]{Notation}
\newtheorem{warn}[thm]{Warning}

\newcounter{dummy}

\newtheorem{hypoAlph}[dummy]{Hypothesis}

\DeclareSymbolFontAlphabet{\scr}{rsfs}
\newcommand{\Fh}{\sF}
\newcommand{\Oh}{\sO}

\author{Annette Huber} %
\address{Annette Huber, Mathematisches Institut, Albert-Ludwigs-Universität
  Freiburg, Eckerstraße 1, 79104 Freiburg im Breisgau, Germany}
\email{\href{mailto:annette.huber@math.uni-freiburg.de}{annette.huber@math.uni-freiburg.de}}

\author{Stefan Kebekus} %
\address{Stefan Kebekus, Mathematisches Institut, Albert-Ludwigs-Universität
  Freiburg, Eckerstraße 1, 79104 Freiburg im Breisgau, Germany and University of
  Strasbourg Institute for Advanced Study (USIAS), Strasbourg, France}
\email{\href{mailto:stefan.kebekus@math.uni-freiburg.de}{stefan.kebekus@math.uni-freiburg.de}}
\urladdr{\url{http://home.mathematik.uni-freiburg.de/kebekus}}

\author{Shane Kelly} %
\address{Shane Kelly, Interactive Research Center of Science, Graduate School of
  Science and Engineering, Tokyo Institute of Technology, 2-12-1 O-okayama,
  Meguro, Tokyo 152-8551 Japan}
\email{\href{mailto:shanekelly64@gmail.com}{shanekelly64@gmail.com}}

\thanks{Stefan Kebekus gratefully acknowledges the support through a joint
  fellowship of the Freiburg Institute of Advanced Studies (FRIAS) and the
  University of Strasbourg Institute for Advanced Study (USIAS).  Shane Kelly
  gratefully acknowledges the support of the Alexander von Humboldt Foundation
  and the DFG through the SFB Transregio 45 via a post-doc in Marc Levine's
  workgroup, and the support of the Japan Society for the Promotion of Science
  via a Grant-in-Aid.}

\title[Differential forms in positive characteristic]{Differential forms in positive characteristic avoiding resolution of singularities}
\date{\today}

\makeatletter
\hypersetup{
  pdfauthor={\authors},
  pdftitle={\@title},
  pdfsubject={\@subjclass},
  pdfkeywords={\@keywords},
  pdfstartview={Fit},
  pdfpagelayout={TwoColumnRight},
  pdfpagemode={UseOutlines},
  bookmarks,
  colorlinks,
  linkcolor=linkblue,
  citecolor=linkred,
  urlcolor=linkred
}
\makeatother

\begin{document}

\maketitle

\begin{abstract}
  This paper studies several notions of sheaves of differential forms that are
  better behaved on singular varieties than Kähler differentials.  Our main
  focus lies on varieties that are defined over fields of positive
  characteristic.  We identify two promising notions: the sheafification with
  respect to the cdh-topology, and right Kan extension from the subcategory of
  smooth varieties to the category of all varieties.  Our main results are that
  both are cdh-sheaves and agree with Kähler differentials on smooth
  varieties.  They agree on all varieties under weak resolution of singularities.
 
  A number of examples highlight the difficulties that arise with torsion forms
  and with alternative candiates.
\end{abstract}

\approvals{
  Annette & yes \\
  Shane & yes \\
  Stefan & yes
}

\setcounter{tocdepth}{1}
\tableofcontents

%
%
\svnid{$Id: 01.tex 246 2015-02-23 09:33:16Z kebekus $}

\section{Introduction}
\label{sec:intro}
\subversionInfo
\approvals{
  Annette & yes \\
  Shane & yes \\
  Stefan & yes
}

Sheaves of differential forms play a key role in many areas of algebraic and
arithmetic geometry, including birational geometry and singularity theory.  On
singular schemes, however, their usefulness is limited by bad behaviour such as
the presence of torsion sections.  There are a number of competing modifications
of these sheaves, each generalising one particular aspect.  For a survey see the
introduction of \cite{HJ}.

In this article we consider two modifications, $Ω^n_{\cdh}$ and $Ω_{\rs}^n$, to
the presheaves $Ω^n$ of relative $k$-differentials on the category $\Sch(k)$ of
separated finite type $k$-schemes.  By $Ω_{\cdh}^n$ we mean the sheafification
of $Ω^n$ with respect to the cdh topology, cf.~Definition~\ref{def:cdhdiffs},
and by $Ω_{\rs}^n$ we mean the right Kan extension along the inclusion $\Sm(k) →
\Sch(k)$ of the restriction of $Ω^n$ to the category $\Sm(k)$ of smooth
$k$-schemes, cf.~Definition~\ref{defi:rs-differentials}.

The following are three of our main results.

\begin{thm}\ \label{1}
  Let $k$ be a perfect field and $n ≥ 0$.
  \begin{enumerate}
  \item \label{theo1:part:cdhregular} (Theorem~\ref{prop:cdhDescentModTorsion}).
    If $X$ is a smooth $k$-variety then $Ω^n(X) \cong Ω^n_{\cdh}(X)$.  The same
    is true in the rh- or eh-topology.
  \item \label{theo1:part:cdhrsmono} (Observation~\ref{obs:oisehs},
    Proposition~\ref{prop:cdhInRs}).  $Ω^n_{\rs}$ is a $\cdh$-sheaf and the
    canonical morphism
    $$
    Ω_{\cdh}^n → Ω^n_{\rs}
    $$
    is a monomorphism.

  \item \label{theo1:part:cdhrsiso} (Proposition~\ref{prop:cdhRsAgree}).  Under
    weak resolution of singularities, this canonical morphism is an isomorphism
    $Ω_{\cdh}^n \cong Ω^n_{\rs}$.  \qed
  \end{enumerate}
\end{thm}

Part~\ref{theo1:part:cdhregular} was already observed by Geisser, assuming a
strong form of resolution of singularities, \cite{Gei06}.  We are able to give a
proof which does not assume any conjectures.  The basic input into the proof is
a fact about torsion forms (Theorem~\ref{hypH}): given a torsion form on an
integral variety, there is a blow-up such the pull-back of the form vanishes on
the blow-up.

\subsection{Comparison to known results in characteristic zero}
\approvals{
  Annette & yes \\
  Shane & yes \\
  Stefan & yes
}

This paper aims to extend the results of \cite{HJ} to positive characteristic,
avoiding to assume resolution of singularities if possible.  The following
theorem summarises the main results known in characteristic zero.

\begin{theo}[{\cite{HJ}}]\label{theo:HJ}
  Let $k$ be a field of characteristic zero, $X$ a separated finite type
  $k$-scheme, and $n ≥ 0$.
  \begin{enumerate}
  \item \label{HJvectorbundle} The restriction of $Ω^n_\h$ to the small Zariski
    site of $X$ is a torsion-free coherent sheaf of $\mathcal{O}_X$-modules.

  \item \label{HJtorsionfree} If $X$ is reduced we have
    $$
    Ω^n(X) / \{ \textrm{ torsion elements }\} ⊆ Ω_\h^n(X)
    $$
    and if $X$ is Zariski-locally isomorphic to a normal crossings divisor in a
    smooth variety then
    $$
    Ω^n(X) / \{ \textrm{ torsion elements }\} \cong Ω_\h^n(X).
    $$

  \item \label{HJsmoothcohomology} If $X$ is smooth, then $Ω^n(X) \cong
    Ω_\h^n(X)$ and $H_{Zar}^i(X, Ω^n) \cong H_\h^i(X, Ω_\h^n)$ for all $i ≥ 0$.
    The same is true using the $\cdh$- or $\eh$-topology in place of the
    $\h$-topology.

  \item \label{HJKanextension} We have $Ω_{\rs}^n \cong Ω^n_h$,
    cf.~Definition~\ref{defi:rs-differentials}.  \qed
  \end{enumerate}
\end{theo}

\subsubsection*{Failure of \ref*{HJvectorbundle} and \ref*{HJtorsionfree} in positive characteristic}
\approvals{
  Annette &yes \\
  Shane & yes \\
  Stefan & yes
}

In positive characteristic, the first obstacle to this program one discovers is
that $Ω_\h^n = 0$ for $n≥ 1$, cf.~Lemma~\ref{lemm:hNoGood}.  This is due to the
fact that the geometric Frobenius is an $\h$-cover, which induces the zero
morphism on differentials.  However, almost all of the results of \cite{HJ} are
already valid in the coarser $\cdh$-topology, and remain valid in positive
characteristic if one assumes that resolutions of singularities exist.  So let
us use the $\cdh$-topology in place of the $\h$-topology.  But even then,
\ref{HJvectorbundle} and \ref{HJtorsionfree} of Theorem~\ref{theo:HJ} seem to be
lost causes:

\begin{coro}[Corollary~\ref{coro:rs_not_torsionfree2}, Corollary~\ref{coro:torsionIsNotAPresheafForcdh}, Example~\ref{exam:pullbackTorsionNotTorsion}]
  For perfect fields of positive characteristic, there exist varieties $X$ such
  that the restriction of $Ω^1_{\cdh}$ to the small Zariski site of $X$ is not
  torsion-free.

  Moreover, there exist morphisms $Y → X$ and torsion elements of
  $Ω^1_{\cdh}(X)$ (resp.  $Ω^1(X)$) whose pull-back to $Ω^1_{\cdh}(Y)$ (resp.
  $Ω^1(Y)$) are not torsion.  \qed
\end{coro}

Note that functoriality of torsion forms over the complex numbers is true,
cf.~Theorem~\ref{thm:pb-torDiff-cplx}, \cite[Corollary~2.7]{MR3084424}.

\subsubsection*{Positive results}
\approvals{
  Annette & yes \\
  Shane & yes \\
  Stefan & yes }

On the positive side, Item~\ref{theo1:part:cdhregular} in Theorem~\ref{1} can be
seen as an analogue of Item~\ref{HJsmoothcohomology} in Theorem~\ref{theo:HJ}.
In particular, we can give an unconditional statement of the case $i = 0$.  In a
similar vein, Items~\ref{theo1:part:cdhrsmono} and \ref{theo1:part:cdhrsiso} of
Theorem~\ref{1} relate to Item~\ref{HJKanextension} in Theorem~\ref{theo:HJ}.

\subsection{Other results}
\approvals{
  Annette & yes \\
  Shane & yes \\
  Stefan & yes }

Many of the properties of $Ω^n_{\rs}$ hold for a more general class of
presheaves, namely unramified presheaves, introduced by Morel,
cf.~Definition~\ref{def:unr}.  The results mentioned above are based on the
following very general result which should be of independent interest.

\begin{prop}[Proposition~\ref{lemm:isrh}]
  Let $S$ be a Noetherian scheme.  If $\Fh$ is an unramified presheaf on
  $\Sch(S)$ then $\Fh_{\rs}$ is an $\rh$-sheaf.  In particular, if $\Fh$ is an
  unramified Nisnevich (resp.~étale) sheaf on $\Sch(S)$ then $\Fh_{\rs}$ is a
  $\cdh$-sheaf (resp.~$\eh$-sheaf).  \qed
\end{prop}

In our effort to avoid assuming resolution of singularities, we investigated the
possibility of a topology sitting between the $\cdh$- and $\h$-topologies which
might allow the theorems of de Jong or Gabber on alterations to be used in place
of resolution of singularities.  An example of the successful application of
such an idea is \cite{Kel12} where the $l$dh-topology is introduced and
successfully used as a replacement to the $\cdh$-topology.
Section~\ref{section:newTopologies} proposes a number of new, initially
promising sites, cf.~Definitions~\ref{defi:sdhtopology} and \ref{defi:saltSite},
but then also shows in Example~\ref{exam:cdhDecsentFailure} that, somewhat
surprisingly, the sheafification of $Ω^1$ on these sites does not preserve its
values on regular schemes, cf.~Proposition~\ref{prop:noSdhDescent} and
Lemma~\ref{lemm:nosaltdescent}.

\subsection{Outline of the paper}
\approvals{
  Annette & yes \\
  Shane & yes\\
  Stefan & yes
}

After fixing notation in Section~\ref{sec:02}, the paper begins in
Section~\ref{section:torsionFree} with a discussion of torsion- and torsion-free
differentials.  Section~\ref{sec:extension} contains a general discussion of the
relevant properties of unramified presheaves, whereas properties that are
specific to $Ω^1$ are collected in Section~\ref{section:rs-diff}.
Section~\ref{section:rs-diff} discusses our proposals for a good presheaf of
differentials on singular schemes---$Ω_{\cdh}^1$ and $Ω_{\rs}^1$---and their
properties.  Appendix~\ref{appA} gives the proof of the above mentioned result
on killing torsion forms by blow-up.  We also discuss a hyperplane section
criterion for testing the vanishing of torsion forms.

\subsection{Open problems}
\approvals{
  Annette & yes \\
  Shane & yes \\
  Stefan & yes
}

What is missing from this paper is a full $\cdh$-analogue of
Theorem~\ref{theo:HJ}, Item~\ref{HJsmoothcohomology}.  Assuming resolutions of
singularities, Geisser has shown \cite{Gei06} that the $\cdh$-cohomology of
$Ω^n_{\cdh}$ agrees with Zariski-cohomology of $Ω^n$ on all smooth varieties
$X$.  It remains open if this can be extended unconditionally to $Ω^n_{\cdh}$
and $Ω^n_{\rs}$.

In a similar vein, we do not know if the assumption on resolutions of
singularities can be removed from Item~\ref{theo1:part:cdhrsiso} of
Theorem~\ref{1}.

\subsection{Acknowledgements}
\approvals{
  Annette & yes \\
  Shane & yes \\
  Stefan & yes
}

The authors thank Daniel Greb for stimulating discussions about torsion-forms,
and Orlando Villamayor as well as Mark Spivakovsky for answering questions
concerning resolution of singularities.  We are most indebted to a referee of an
earlier version of this paper for pointing us to the work of Gabber-Ramero.

%
%
\svnid{$Id: 02.tex 236 2015-02-20 13:00:44Z kebekus $}

\section{Notation and Conventions}
\label{sec:02}
\subversionInfo
\subsection{Global assumptions}
\approvals{
  Annette & yes \\
  Shane & yes \\
  Stefan & yes
}

Throughout the present paper, all schemes are assumed to be separated.  The
letter $S$ will always denote a fixed, separated, noetherian base scheme.

\subsection{Categories of schemes and presheaves}
\approvals{
  Annette & yes \\
  Shane & yes \\
  Stefan & yes
}

Denote by $\Sch(S)$ the category of separated schemes of finite type over $S$,
and let $\Reg(S)$ be the full subcategory of regular schemes in $\Sch(S)$.  If
$S$ is the spectrum of a field $k$, we also write $\Sch(k)$ and $\Reg(k)$.  If
$k$ is perfect, then $\Reg(k)$ is the category of smooth $k$-varieties, which
need not necessarily be connected.

\begin{nota}[Presheaf on $\Sch(S)$]
  Given a noetherian scheme $S$, a presheaf $\Fh$ of abelian groups on $\Sch(S)$
  is simply a contravariant functor
  $$
  \Sch(S) → \{\text{abelian groups}\}.
  $$
  A presheaf $\Fh$ is called a presheaf of $\sO$-modules, if every $\Fh(X)$ has
  an $\sO_X(X)$-module structure such that for every morphism $Y → X$ in
  $\Sch(S)$, the induced maps $\Fh(X) → \Fh(Y)$ are compatible with the
  morphisms of functions $\sO(X) → \sO(Y)$.
\end{nota}

We are particularly interested in the presheaf of Kähler differentials.

\begin{exam}[Structure sheaf, Kähler differentials]
  We denote by $\sO$ the presheaf $X \mapsto \sO_{X}(X)$.  Given any $n ∈ \bN^{≥
    0}$, we denote by $Ω^n$ the presheaf $X \mapsto Ω^n_{X/S}(X)$.  Note that
  $Ω^0 = \sO$.  We also abbreviate $Ω = Ω¹$.

  The notation $Ω^n_X$ means the usual Zariski-sheaf on $X$.  That is, $Ω^n_X =
  Ω^n_{\vphantom{X}}|^{\vphantom{n}}_{X_{\Zar}}$ where $X_{\Zar}$ is the usual
  topological space associated to the scheme $X$.
\end{exam}

\begin{defi}[Torsion]\label{defi:torsionF}
  Let $\Fh$ be a presheaf on $\Sch(S)$ and $X∈ \Sch(S)$.  We write $\tor \Fh(X)$
  for the set of those sections of $\Fh(X)$ which vanish on a dense open
  subscheme.
\end{defi}

\begin{warn}[Torsion groups might not form a presheaf]
  In the setting of Definition~\ref{defi:torsionF} note that the groups
  $\tor\Fh(X)$ do not necessarily have the structure of a presheaf on $\Sch(S)$!
  For a morphism $Y → X$ in $\Sch(S)$, the image of $\tor \Fh(X)$ under the
  morphism $\Fh(X) → \Fh(Y)$ does not necessarily lie in $\tor \Fh(Y)$.  A very
  simple example (pointed out to us by a referee) is the section $x$ on $\Spec
  k[x,y]/(x^2,y)$.  It is torsion, but its pull-back to $\Spec k k[x]/x^2$ is
  not.  Example~\vref{exam:pullbackTorsionNotTorsion} shows that the problem
  also happens for reduced varieties in positive characteristic in the case
  where $\Fh = Ω$.  On the positive side, note that the restriction of a torsion
  section to an open subscheme is again a torsion section.
\end{warn}

\subsection{Schemes and morphisms of essentially finite type}
\approvals{
  Annette & yes \\
  Shane & yes \\
  Stefan & yes
}

Parts of Section~\ref{sec:extension} use the notion of ``$S$-schemes that are
essentially of finite type''.  While this notion has been used at several places
in the literature, we were not able to find a convenient reference for its
definition.  We have therefore chosen to include a definition and a brief
discussion here.

\begin{defi}\label{defi:ess_ft}
  We say that an $S$-scheme $X'$ is \emph{essentially of finite type} over $S$
  if there is a scheme $X$ of finite type over $S$ and a filtred inverse system
  $\{U_i\}_{i∈ I}$ of open subschemes of $X$ with affine transition maps such
  that $X' = \bigcap_{i∈ I}U_i$.
\end{defi}

\begin{lemm}[Morphisms of schemes essentially of finite type]\label{lem:mor_ft}
  Let $X$ and $Y$ be in $\Sch(S)$.  Let $\{U_i\}_{i∈ I}$ be a filtred inverse
  system of open subschemes of $X$ with affine transitions maps with
  intersection $X'$ and $\{V_j\}_{j∈ J}$ a filtred inverse system of open
  subschemes of $Y$ with affine transition maps with intersection $Y'$.  Then
  \[
  \Mor_S(X',Y') =\varprojlim_j\varinjlim_i\Mor_S(U_i,V_j)\ .
  \]
\end{lemm}

\begin{rema}\label{rem:lem:mor_ft}
  Using the language of pro-categories, briefly recalled in
  Section~\ref{ssec:dvr}, Lemma~\ref{lem:mor_ft} asserts that the category of
  schemes essentially of finite type over $S$ is a full subcategory of the
  pro-category of $\Sch(S)$.
\end{rema}

\begin{proof}[Proof of Lemma~\ref{lem:mor_ft}]
  This is just a special case of \cite[Corollaire~8.13.2]{EGAIV3}.  The key
  point of the argument is that a morphism with source Spec of the local ring of
  a variety always extends to an open neighbourhood.
\end{proof}

\begin{exam}\label{ex:cofinal}
  Let $X ∈ \Sch(S)$ and $\mathfrak{U}=\{U_i\}_{i∈ I}$ be a filtred inverse
  system of open affine subschemes of $X$ and with intersection $X'$.  Let $U$
  be an open affine neighbourhood of $X'$ in $X$.  By Lemma~\ref{lem:mor_ft}, we
  have
  \[
  \Mor_X(X',U)=\varinjlim_i\Mor_X(U_i,U).
  \]
  Hence the inclusion $X'→ U$ factors via some $U_i$.  This means that
  $\mathfrak{U}$ is cofinal in the system of all affine open neighbourhoods of
  $X'$.
\end{exam}

\subsection{Topologies}
\approvals{
  Annette & yes \\
  Shane & yes \\
  Stefan & yes
}

We are going to use various topologies on $\Sch(S)$, which we want to introduce
now.  They are variants of the $\h$-topology introduced by Voevodsky in
\cite{Voev96}.  Recall that a Grothendieck topology on $\Sch(S)$ is defined by
specifying for each $X ∈ \Sch(S)$ which collections $\{U_i→ X\}_{i∈ I}$ of
$S$-morphisms should be considered as open covers.  By definition, a presheaf
$\Fh$ is a sheaf if for any such collection, $\Fh(X)$ is equal to the set of
those elements $(s_i)_{i ∈ I}$ in $\prod_{i ∈ I}\Fh(U_i)$ for which $s_i|_{U_i
  {⨯_X} U_j} = s_j|_{U_i {⨯_X} U_j}$ for every $i, j ∈ I$.

We refer to the ordinary topology as the Zariski topology.

\begin{defi}[$\cdp$-morphism]\label{defi:cdp}
  A morphism $f : Y → X$ is called a \emph{$\cdp$-morphism} if it is proper and
  completely decomposed, where by ``completely decomposed'' we mean that for
  every point $x ∈ X$ there is a point $y ∈ Y$ with $f(y) = x$ and $[k(y): k(x)]
  = 1$.
\end{defi}

These morphisms are also referred to as proper \emph{$\cdh$-covers}, or
\emph{envelopes} in the literature.

\begin{rema}[$\rh$, $\cdh$ and $\eh$-topologies]\label{rem:defi:rhetc}
  Recall that the $\rh$-topology on $\Sch(S)$ is generated by the Zariski
  topology and $\cdp$-morphisms, \cite{GL}.  In a similar vein, the
  $\cdh$-topology is generated by the Nisnevich topology and $\cdp$-morphisms,
  \cite[§~5]{SV00}.  The $\eh$-topology is generated by the étale topology and
  $\cdp$-morphisms, \cite{Gei06}.
\end{rema}

We are going to need the following fact from algebraic geometry.

\begin{lemm}\label{lem:dominateByR}
  Let $X$ be a regular, noetherian scheme.  Let $x ∈ X$ be a point of
  codimension $n$.  Then there is point $y∈ X$ of codimension $n-1$, a discrete
  valuation ring $R$ essentially of finite type over $X$ together with a map
  $\Spec\, R → X$ such that the special point of $\Spec\, R$ maps to $x$ and the
  generic point to $y$, both inducing isomorphisms on their respective residue
  fields.
\end{lemm}
\begin{proof}
  The local ring $\sO_{X, x}$ is a regular local ring, and as such admits a
  regular sequence $f_1, \dots, f_n$ generating its maximal ideal.  The quotient
  ring
  $$
  R := \factor{\sO_{X, x}}{\langle f_1, \dots, f_{n - 1}\rangle}
  $$
  is then a regular local ring of dimension one, that is, a discrete valuation
  ring, \cite[Theorems~36(3) and 17.G]{Mat}.  Let $y$ be the image of the
  generic point of $\Spec\, R$.  By construction, this is a point of codimension
  $n-1$.
\end{proof}

\begin{propo}[Birational-- and $\cdp$-morphisms]\label{prop:cdpRefinable}
  Suppose that $X$ is a regular, noetherian scheme.  Then, every proper,
  birational morphism is a $\cdp$-morphism, and every $\cdp$-morphism is
  refinable by a proper, birational morphism.
\end{propo}

\begin{rema}
  This fact is well-known over a field of characteristic zero and is usually
  proven using strong resolution of singularities.  That is, by refining a
  proper birational morphism by a sequence of blow-ups with smooth centres.  By
  contrast, the proof below works for \emph{any} regular, noetherian scheme $X$,
  without restriction on a potential base scheme, or structural morphism.
\end{rema}

\begin{proof}[Proof of Proposition~\ref{prop:cdpRefinable}]
  With $X$ as in Proposition~\ref{prop:cdpRefinable}, we show the two statements
  separately.

  \subsubsection*{Step 1: Birational morphisms are $\cdp$}

  Let $Y → X$ be proper and birational.  We must show that for every point $x ∈
  X$ the canonical inclusion admits a factorisation $x → Y → X$.  We proceed by
  induction on the codimension.  In codimension zero, the factorisation is a
  consequence of birationality.  Suppose that it is true up to codimension $n -
  1$ and let $x$ be a point of codimension $n$.  By Lemma~\ref{lem:dominateByR}
  we can find a discrete valuation ring $R$ and a diagram
  $$
  \xymatrix@R=0pt{
    y \ar[dr] & \\
    & \Spec\, R \ar[r] & X \\
    x \ar[ur]
  }
  $$
  for some $y$ of codimension $n - 1$.  By the inductive hypothesis, the
  inclusion of $y$ into $X$ admits a factorisation through $Y$, so we have a
  commutative diagram
  $$
  \xymatrix{
    y \ar[r] \ar[d] & Y \ar[d] \\
    \Spec\, R \ar[r] & X
  }
  $$
  and now the valuative criterion for properness implies that the inclusion of
  $\Spec\, R$ into $X$ factors trough $Y$, and therefore so does the inclusion
  of $x$.

  \subsubsection*{Step 2: $\cdp$-morphisms are refinable}
  If $Y → X$ is a proper, completely decomposed morphism with $X$ connected
  (hence irreducible), choose a factorisation $η → Y → X$ of the inclusion of
  the generic point $η$ of $X$.  Then, the closure of the image of $η$ in $Y$ is
  birational and proper over $X$.
\end{proof}

\begin{lemm}\label{lem:makecover}
  Let $Y ∈ \Sch(S)$ be an integral scheme and let $\{U_i → Y\}_{i∈ I}$ be an étale
  cover of $Y$ by finitely many integral schemes.  Assume further that for each
  for $i ∈ I$ we are given a proper, birational morphism $T_i→ U_i$.  Then there
  exists a proper, birational morphism $Y'→ Y$, an étale cover $\{U'_i → Y'\}_{i
    ∈ I}$ and for each $i ∈ I$ a commutative diagram of the following form,
  $$
  \xymatrix{ %
    U'_i \ar[rrrr]^{\text{étale}} \ar[d]_{\text{proper, biratl.}} &&&& Y' \ar[d]^{\text{proper, biratl.}} \\
    T_i \ar[rr]_{\text{proper, biratl.}} && U_i \ar[rr]_{\text{étale}} && Y
  }
  $$
  If $\{U_i→ Y\}_{i∈ I}$ is a Nisnevich-- or a Zariski-cover, then so is the
  cover $\{U'_i → Y'\}_{i∈ I}$.
\end{lemm}
\begin{proof}
  It follows from flattening by blow-up, \cite[Théorème~5.2.2]{RG}, that there
  exists an integral scheme $Y'$ and a proper birational morphism $Y' → Y$ such
  that for any $i ∈ $, the strict transforms $T'_i → Y'$ of the $T_i → Y$ are flat.
  These morphisms factor via the pullbacks $U'_i = U_i ⨯_Y Y'$ of the $U_i$,%
  $$
  T'_i → U'_i → Y.
  $$
  Since $\{U_i → Y\}_{i ∈ I}$ is an étale (resp.\ Zariski, Nisnevich) cover, so
  is $\{U'_i → Y'\}_{i ∈ I}$.

  It remains to show that $T'_i \cong U'_i$.  The scheme $U'_i$ is integral
  because it is proper and birational over $U_i$.  The scheme $T'_i$ is integral
  because it is proper and birational over $T_i$.  As $T'_i → Y'$ is flat and
  $U'_i→ Y'$ is unramified, the morphism $T'_i → U'_i$ is flat by
  \cite[Proposition~17.7.10]{EGAIV4}.  We now have a flat, proper and birational
  morphism between integral schemes, hence an isomorphism.  In detail: by
  flatness, the morphism $T'_i → U'_i$ has constant fibre dimension, which, by
  birationality, equals zero.  This means that the morphism is quasi-finite.  As
  it is also proper, this means that it is finite and still flat.  It follows
  that $\sO_{T'_i}$ is a locally free $\sO_{U'_i}$-module of constant rank.  The
  rank is one because again the morphism is birational.
\end{proof}

\begin{coro}[Normal form]\label{cor:normalform}
  Let $Y ∈ \Sch(S)$, and let $\{Y_i → Y\}_{i ∈ I}$ be an $\rh$-, $\cdh$- or
  $\eh$-cover, respectively.  Then there exists a refinement of the following
  form,
  $$
  \{ Y'_i \longrightarrow Y' \xrightarrow{\text{$\cdp$-covering}} Y \}_{i ∈ I},
  $$
  where $\{Y'_i → Y'\}_{i∈ I}$ is a Zariski-, Nisnevich- or étale cover,
  respectively, with $I$ finite.  If $Y$ is regular, then we can even assume
  that $Y' → Y$ is a proper birational morphism.
\end{coro}
\begin{proof}
  In the $\cdh$-case, this is precisely \cite[Prop.~5.9]{SV00}.  In the two
  other cases, the same argument works using noetherian induction, and using
  Lemma~\ref{lem:makecover} in the appropriate place.  The last claim uses
  Proposition~\ref{prop:cdpRefinable}.
\end{proof}

%
%
\svnid{$Id: 03.tex 237 2015-02-20 13:02:26Z kebekus $}

\section{(Non-)Functoriality of torsion-forms}
\label{section:torsionFree}
\subversionInfo
\approvals{
  Annette & yes \\
  Shane & yes \\
  Stefan & yes
}

One very useful feature of differential forms on a smooth varieties is that they
are a vector bundle, in particular torsion-free.  In characteristic zero, the
different candidates for a good theory of differential forms share this
behaviour on all varieties.  It is disappointing but true that this property
\emph{fails} in positive characteristic, as we are going to establish.  The
following notion will be used throughout.

\begin{defi}[\protect{Torsion-differentials and torsion-free differentials, \cite[Section~2.1]{MR3084424}}]\label{defi:torsionfree}
  Let $k$ be a field and $X ∈ \Sch(k)$.  We define the sheaf $\check{Ω}^n_X$ on
  $X_{\Zar}$ as the cokernel in the sequence
  \begin{equation}\label{eq:torSeq}
    \xymatrix{ %
      0 \ar[r] & \tor Ω^n_X \ar[r]^{ α_X} & Ω^n_X \ar[r]^{β_X} & \check{Ω}^n_X \ar[r] & 0 }
  \end{equation}
  Sections in $\tor Ω^n_X$ are called \emph{torsion-differentials}.  By slight
  abuse of language, we refer to sections in $\check{Ω}^n_X$ as
  \emph{torsion-free differentials}.
\end{defi}

\begin{rema}[Torsion-sheaves on reducible spaces]
  Much of the literature discusses torsion-sheaves and torsion-free sheaves only
  in a setting where the underlying space is irreducible.  We refer to
  \cite[Appendix~A and references there]{MR3084424} for a brief discussion of
  torsion-sheaves on reduced, but possibly reducible spaces.
\end{rema}

\subsection{Torsion-free forms over the complex numbers}
\approvals{
  Annette & yes \\
  Shane & yes\\
  Stefan & yes
}

Given a morphism between two varieties that are defined over the complex
numbers, the usual pull-back map of Kähler differentials induces pull-back maps
for torsion-differentials and for torsion-free differentials, even if the image
of the morphism is contained in the singular set of the target variety.

\begin{theo}[\protect{Pull-back for sheaves of torsion-free differentials, \cite[Corollary~2.7]{MR3084424}}]\label{thm:pb-torDiff-cplx}
  Let $f : X → Y$ be a morphism of reduced, quasi-projective schemes that are
  defined over the complex numbers.  Then there exist unique morphisms $d_{\tor}
  f$ and $\check df$ such that the following diagram, which has exact rows,
  becomes commutative
  $$
  \xymatrix{ %
    & f^* \tor Ω^n_Y \ar[rr]^{f^* α_Y} \ar[d]_{d_{\tor} f} && f^* Ω^n_Y \ar[rr]^{f^* β_Y} \ar[d]_{df} && f^* \check{Ω}^n_Y \ar[r] \ar[d]^{\check df} & 0 \\
    0 \ar[r] & \tor Ω^n_X \ar[rr]_{α_X} && Ω^n_X \ar[rr]_{β_X} && \check{Ω}^n_X \ar[r] & 0.  &
  }
  $$
  In other words, $\torΩ^n$ is a presheaf on $\Sch(\bC)$.  \qed
\end{theo}

The same argument works for any field of characteristic zero.

\begin{rema}[Earlier results]
  For complex spaces, the existence of a map $\check d$ has been shown by
  Ferrari, \cite[Proposition~1.1]{Ferr70}, although it is perhaps not obvious that the
  sheaf discussed in Ferrari's paper agrees with the sheaf of Kähler
  differentials modulo torsion.
\end{rema}

\begin{warn}[Theorem~\ref{thm:pb-torDiff-cplx} is wrong in the relative setup]
  One can easily define torsion-differentials and torsion-free differentials in
  the relative setting.  The proof of Theorem~\ref{thm:pb-torDiff-cplx},
  however, relies on the existence of a resolution of singularities for which no
  analogue exists in the relative case.  As a matter of fact,
  Theorem~\ref{thm:pb-torDiff-cplx} becomes wrong when working with relative
  differentials, unless one makes rather strong additional assumptions.  A
  simple example is given in \cite[Warning~2.6]{MR3084424}.
\end{warn}

\subsection{Torsion-free forms in positive characteristic}
\approvals{
  Annette & yes\\
  Shane & yes\\
  Stefan & yes
}

Now let $f : X → Y$ be a morphism of reduced, quasi-projective schemes that are
defined over a field of finite characteristic.  We will see in this section that
in stark contrast to the case of complex varieties, the pull-back map $df$ of
Kähler differential does generally \emph{not} induce a pull-back map between the
sheaves of torsion-free differential forms.

Indeed, if there exists a pull-back map $\check df : f^* \check{Ω}^n_{Y} →
\check{Ω}^n_X$ which makes the following diagram commute,
$$
\xymatrix{ %
  f^* Ω^n_Y \ar[rr]^{f^* β_Y} \ar[d]_{df} && f^* \check{Ω}^n_Y \ar[d]^{\check df} \\
  Ω^n_X \ar[rr]_{β_X} && \check{Ω}^n_X,
}
$$
and if $σ ∈ \tor Ω¹_Y$ is any torsion-differential, then $df(σ)$ is necessarily
a torsion-differential on $X$, that is $df(σ) ∈ \tor Ω¹_X$.  The following
example discusses a morphism between varieties for which this property does not
hold.

\begin{exam}[Pull-back of torsion-form is generally not torsion]\label{exam:pullbackTorsionNotTorsion}
  Variants of this example work for any prime $p$, but we choose $p = 2$ for
  concreteness.  Let $k$ be an algebraically closed field of characteristic two
  and let $Y ⊂ \bA³_k$ be the Whitney umbrella.  More precisely, consider the
  ring $R := \factor{k[x,y,z]}{(y²-xz²)}$ and the schemes
  $$
  X := \Spec\, k[x] \quad \text{and} \quad Y := \Spec\, R.
  $$
  An elementary computation shows that the polynomial $y²-xz²$ is irreducible.
  As a consequence, we see that $Y$ is reduced and irreducible and that $z$ is
  not a zerodivisor in $R$.  Finally, let $f : X → Y$ be the obvious inclusion
  map, which identifies $X$ with the $x$-axis in $\bA³$, and which is given by
  the following map of rings,
  $$
  f^\# : \factor{k[x,y,z]}{(y²-xz²)} → k[x] \qquad Q(x,y,z) \mapsto Q(x,0,0).
  $$
  Note that $X$ is nothing but the reduced singular locus of $Y$.  We want to
  construct a torsion-differential $σ$ on $Y$.  To this end, notice that the
  differential form $dP ∈ Γ\bigl(Ω¹_{\bA³}\bigr)$, where $P = y²-xz²$, induces
  the zero-form on $Y$,
  $$
  0 = dP = - z² · dx + 2y· dy - 2xz · dz = -z² · dx ∈ Γ \bigl(
  Ω¹_Y \bigr).
  $$
  Since $z²$ is not a zerodivisor, we see that the form $σ := dx$ is torsion,
  that is, $σ ∈ Γ \bigl(\tor Ω¹_Y \bigr)$.  On the other hand, the pull-back of
  $σ$ to $X$ is clearly given by $df (σ) = dx ∈ Γ \bigl( Ω¹_X \bigr)$, which is
  not torsion.
\end{exam}

\begin{summ}\label{summ:torsion}
  Example~\ref{exam:pullbackTorsionNotTorsion} shows that the assignments
  $X\mapsto \torΩ^n(X)$ and $X\mapsto\check{Ω}^n_X(X)$ do not in general define
  presheaves on $\Sch(k)$.
\end{summ}

%
%
\svnid{$Id: 04.tex 239 2015-02-20 14:28:18Z kebekus $}

\section{The extension functor}
\label{sec:extension}
\subversionInfo

\subsection{Definition and first properties} \label{section:functorRs}
\approvals{
  Annette & yes \\
  Shane & yes \\
  Stefan & yes
}

In analogy with the sheaves discussed in the introduction, we aim to define a
``good'' sheaf on $\Sch(S)$, which agrees on $\Reg(S)$ with Kähler differentials
and avoids that pathologies exposed by Kähler differentials on singular schemes.
This section provides the technical framework for one construction in this
direction: ignoring non-regular schemes, we define a sheaf on $\Sch(S)$ whose
value group at one $X ∈ \Sch(S)$ is determined by differential forms on regular
schemes over $X$.  The following definition makes this idea precise.

\begin{defi}[Extension functor]\label{defi:rs1}
  Given a presheaf $\Fh$ on $\Reg(S)$, define a presheaf $\Fh_{\rs}$ on
  $\Sch(S)$ by setting
  $$
  \Fh_{\rs}(X) := \varprojlim_{Y ∈ \Reg(X)} \Fh(Y) \quad \text{for any } X ∈
  \Sch(S).
  $$
  The assignment
  $$
  \begin{array}{ccc}
    \{\text{Presheaves on }\Reg(S) \} & → & \{\text{Presheaves on }\Sch(S) \} \\
    \Fh & \mapsto & \Fh_{\rs}
  \end{array}
  $$
  is clearly functorial, and referred to as the \emph{extension functor}.
\end{defi}

\begin{warn} \label{warn:Freg}
  ---
  \begin{enumerate}
  \item Observe $\Fh_{\rs}(X)$ is defined via limits and not by colimits.

  \item\label{warn:Freg:funct} Given a morphism $X' → X$, the induced morphism
    $\Fh_\rs(X) → \Fh_\rs(X')$ is induced by the composition functor, $\Reg(X')
    → \Reg(X)$.  It is easy to forget this and think that it is induced by $X'
    {⨯_X}-$, which does not necessarily preserve regular schemes.
  \end{enumerate}
\end{warn}

\begin{rema}[Elementary properties of the extension functor]\label{rema:firstRemarks}
  ---
  \begin{enumerate}
  \item Explicitly, a section of $\Fh_{\rs}(X)$ is a sequence of compatible
    sections.  To give a section, it is therefore equivalent to give an element
    $s_Y ∈ \Fh(Y)$ for every morphism $Y → X$ with $Y ∈ \Reg(X)$, such that the
    following compatibility conditions hold: for every triangle $Y' → Y → X$
    with $Y', Y ∈ \Reg(X)$, we have $s_Y|_{Y'} = s_{Y'}$.

  \item The assignment $\Fh \mapsto \Fh_{\rs}$ could equivalently be defined as
    the right adjoint to the restriction functor from presheaves on $\Sch(S)$ to
    presheaves on $\Reg(S)$.

  \item\label{rema:firstRemarks:regularSame} If $X$ itself is regular, then
    $\Fh(X) = \Fh_{\rs}(X)$, since $X$ is then a final object in $\Reg(X)$.

  \item Let $S=\Spec\, k$ with $k$ a field of characteristic zero.  Consider the
    presheaf $\Fh = Ω^n$.  Under these assumptions, it has been shown in
    \cite[Theorem~1]{HJ} that $Ω^n_\rs = Ω^n_\h$.
  \end{enumerate}
\end{rema}

\begin{lemm}[Extension preserves sheaves]\label{lemm:issheaf}
  In the setting of Definition~\ref{defi:rs1}, suppose that $τ$ is a topology on
  $\Sch(S)$ that is equal to or coarser than the étale topology.  Observe that
  $τ$ restricts to a topology $τ_{\rs}$ on $\Reg(S)$.  If $\Fh$ is a
  $τ_{\rs}$-sheaf (respectively, a $τ_{\rs}$-separated presheaf) on $\Reg(S)$,
  then $\Fh_{\rs}$ is a $τ$-sheaf (respectively, a $τ$-separated presheaf) on
  $\Sch(S)$.
\end{lemm}

Lemma~\ref{lemm:issheaf} is in fact a consequence of $\Reg(S) → \Sch(S)$ being a
cocontinuous morphism of sites for such a topology, cf.\
\cite[Definitions~III.2.1 and II.1.2]{SGA41} and \cite[Tags
\href{http://stacks.math.columbia.edu/tag/00XF}{00XF} and
\href{http://stacks.math.columbia.edu/tag/00XI}{00XI}]{stacks-project}.  For the
reader who is not at ease with categorical constructions we give an explicit
proof below.  The key ingredient is the fact that a $τ$-cover of a regular
scheme is regular when $τ$ is coarser than or equal to the étale topology.

\begin{proof}[Proof of Lemma~\ref{lemm:issheaf}]
  We prove the two assertions of Lemma~\ref{lemm:issheaf} separately.

  \subsubsection*{Step 1, $(-)_{\rs}$ preserves separatedness:}

  Assume that $\Fh$ is a $τ_{\rs}$-separated presheaf.  To prove that
  $\Fh_{\rs}$ is $τ$-separated, consider a scheme $X ∈ \Sch(S)$, a $τ$-cover $U
  → X$ and two sections $s, t ∈ \Fh_{\rs}(X)$.  We would like to show that the
  assumption that $s|_U$ and $t|_U$ agree as elements of $\Fh_{\rs}(U)$ implies
  that $s = t$.  Notice that a consequence of the definition of $(-)_{\rs}$ is
  that two sections $s, t ∈ \Fh_{\rs}(X)$ are equal if and only if for every $Y →
  X$ in $\Reg(X)$, their restrictions $s|_Y$ and $t|_Y$ are equal.

  Consider the fibre product diagram,
  $$
  \xymatrix{ %
    U ⨯_X Y =: U_Y \ar@<8mm>[d]_{\text{étale}} \ar[rr] && U \ar[d]^{\text{étale}} \\
    {\hphantom{U ⨯_X Y =: }} Y \ar[rr] && X.
  }
  $$
  Since $U → X$ is étale by assumption, the fibre product $U_Y := U⨯_X Y$ is
  étale over $Y$ (which we are assuming is regular) and therefore $U_Y$ is also
  regular.  By functoriality, equality of $s|_U$ and $t|_U$ implies equality of
  the restrictions $s|_{U_Y}$ and $t|_{U_Y}$, and by
  Remark~\ref{rema:firstRemarks:regularSame} above, we have $\Fh_{\rs}(U_Y) =
  \Fh(U_Y)$ and $\Fh_{\rs}(Y) = \Fh(Y)$.  Since $U_Y$ is a $τ_{\rs}$-cover of
  $Y$, separatedness of $\Fh$ therefore guarantees that the elements $s|_Y, t|_Y
  ∈ \Fh(Y)_{\rs} = \Fh(Y)$ agree.

  \subsubsection*{Step 2, $(-)_{\rs}$ preserves sheaves:}

  Assume that $\Fh$ is a $τ_{\rs}$-sheaf, that $X ∈ \Sch(S)$ and that $U → X$ is
  a $τ$-cover, with connected components $(U_i)_{i ∈ I}$.  Assume further that
  we are given one $s ∈ \Fh_{\rs}(U)$ such that the associated restrictions
  satisfy the compatibility condition
  \begin{equation}\label{eq:xsg}
    s|_{U_i {⨯_X} U_j} = s|_{U_i {⨯_X} U_j} \quad\text{for every } i, j ∈ I.
  \end{equation}
  To prove that $\Fh_{\rs}$ is a $τ$-sheaf, we need to construct a section $t ∈
  \Fh_{\rs}(X)$ whose restriction $t|_U$ agrees with the $s|_U$.  Once $t$ is
  found, uniqueness follows from $τ$-separatedness of $\Fh_{\rs}$ that was shown
  above.

  To give $t ∈ \Fh_{\rs}(X)$, it is equivalent to give compatible elements $t_Y
  ∈ \Fh(Y)$, for every $Y ∈ \Reg(X)$.  Given one such $Y$, consider the base
  change diagram
  $$
  \xymatrix{ %
    (U ⨯_X U)_Y \ar@<-0.5ex>[rr] \ar@<0.5ex>[rr] \ar[d] && U_Y \ar[d] \ar[rr]^{\text{étale}} && Y \ar[d] \\
    U ⨯_X U \ar@<-0.5ex>[rr] \ar@<0.5ex>[rr] && U \ar[rr]_{\text{étale}} && X.
  }
  $$
  As before, observe that $U_Y = U ⨯_X Y$ and $(U ⨯_X U)_Y = U ⨯_X U ⨯_X Y$ are
  regular, that $\Fh_{\rs}(U_Y) = \Fh(U_Y)$ and $\Fh_{\rs}((U {⨯_X} U)_Y) =
  \Fh((U {⨯_X} U)_Y)$, and that $s|_{U_Y} ∈ \Fh(U_Y)$ satisfies a compatibility
  condition, analogous to \eqref{eq:xsg}.  Since $\Fh$ is a $τ_{\rs}$-sheaf, we
  obtain an element $t_Y ∈ \Fh(Y) = \Fh_{\rs}(Y)$ such that $t_Y|_{U_Y} =
  s|_{U_Y}$.  The elements $t_Y$ so constructed clearly define an element $t ∈
  \Fh_{\rs}(X)$ where $t|_U = s|_U$.
\end{proof}

\subsection{Unramified presheaves}
\approvals{
  Annette & yes \\
  Shane & yes\\
  Stefan & yes
}

We will see in Section~\ref{sec:44} that the extension functor admits a
particularly simple description whenever the presheaf $\Fh$ is
\emph{unramified}.  The notion is due to Morel.  We refer the reader to
\cite[Definition~2.1 and Remarks~2.2, 2.4]{Mor12} for a detailed discussion of
unramified presheaves, in the case where the base scheme $S$ is the spectrum of
a perfect field.

\begin{defi}[Unramified presheaf]\label{def:unr}
  A presheaf $\Fh$ on $\Reg(S)$ is \emph{unramified} if the following axioms are
  satisfied for all $X, Y ∈ \Reg(S)$.
  \begin{enumerate}[label=(UNR\arabic{enumi})]
  \item\label{unr0} The canonical morphism $\Fh(X \amalg Y) → \Fh(X) ⨯ \Fh(Y)$
    is an isomorphism.

  \item\label{unr1} If $U → X$ is a dense open immersion, then $\Fh(X) → \Fh(U)$
    is injective.

  \item\label{unr2} The presheaf $\Fh$ is a Zariski sheaf, and for every open
    immersion $U → X$ which contains all points of codimension $≤ 1$ the
    morphism $\Fh(X) → \Fh(U)$ is an isomorphism.
  \end{enumerate}
  We will say that a presheaf $\Fh$ on $\Sch(S)$ is unramified if its
  restriction to $\Reg(S)$ is unramified.
\end{defi}

\begin{exam}\label{ex:unramified}
  ---
  \begin{enumerate}
  \item The sheaf $\sO$ is unramified \cite[Theorem~38, page~124]{Mat}.

  \item If $\Fh$ is a presheaf on $\Reg(S)$ whose restrictions $\Fh|_{X_\Zar}$
    to the small Zariski sites $X|_{\Zar}$ of each $X ∈ \Reg(S)$ are locally
    free, coherent $\sO_X$-modules, then $\Fh$ is unramified.

  \item If $S = \Spec\, k$ is the spectrum of a perfect field, then the
    presheaves $Ω^n$ are unramified, for all $n ≥ 0$.

  \item There are other, important examples of unramified presheaves, which fall
    out of the scope of the article.  These include the Zariski sheafifications
    of $K$-theory, étale cohomology with finite coefficients (prime to the
    characteristic), or homotopy invariant Nisnevich sheaves with transfers.
    More generally, all reciprocity sheaves in the sense of \cite{KSY} which are
    Zariski sheaves satisfy~\ref{unr1} and conjecturally satisfy \ref{unr2}, by
    \cite[Theorem~6 and Conjecture~1]{KSY}.
 \end{enumerate}
\end{exam}

\subsection{Discrete valuation rings}
\label{ssec:dvr}
\approvals{
  Annette & yes \\
  Shane & yes\\
  Stefan & yes
}

If $X ∈ \Sch(S)$ is any scheme over $S$ and $x ∈ X$ any point, we consider
$\Spec\, \sO_{X, x}$, which can be seen as the intersection of (usually
infinitely many) open subschemes $U_i ⊆ X$.  Hence it is an example of a scheme
essentially of finite type over $S$.

\begin{defi}[The category $\Cur(S)$]\label{def:essFT}
  The category of schemes essentially of finite type over $S$ will be denoted by
  $\Sch(S)\ess$.  It contains $\Sch(S)$ as a subcategory.  Let $\Cur(S)$ be the
  category of schemes essentially of finite type which are regular, local, and
  of dimension $≤ 1$.
\end{defi}

\begin{rema}[A description of $\Cur(S)$]
  The schemes in $\Cur(S)$ are of the form $\Spec\, R$ with $R$ either a field
  or a discrete valuation ring.  More precisely, $R=\sO_{X,x}$ where $X ∈
  \Reg(S)$ and $x ∈ X$ is a point of codimension one or zero.  The latter
  amounts to a scheme of the form $\Spec\, K$ for a field extension $K / k(s)$
  of finite transcendence degree of the residue field $k(s)$ of a point $s ∈ S$.
\end{rema}

Any presheaf $\Fh$ on $\Sch(S)$ extends to a presheaf on the larger category
$\Sch(S)\ess$, in a canonical way\footnote{In contrast to $(-)_{\rs}$ however,
  $(-)\ess$ is a \emph{left} Kan extension as opposed to a right Kan extension
  and so the limits in the definition are \emph{co}limits instead of (inverse)
  limits.  That is, instead of a section being described as a coherent sequence
  of sections, it is given by an equivalence class of sections.  Recall the
  difference $\Pi X_i$ vs $\amalg X_i$, or $\varprojlim_{n} \bZ / p^n = \bZ_p$
  vs $\varinjlim_{n} \bZ / p^n = \{ e^{nπ i / p^k} : n, k ∈ \bZ \} ⊂ \bC$.}.

\begin{propdef}[Extension from $\Sch(S)$ to $\Sch(S)\ess$]\label{defi:rs2}
  Given a presheaf $\Fh$ on $\Sch(S)$, define a presheaf $\Fh\ess$ on
  $\Sch(S)\ess$ as follows.  Given any $X' ∈ \Sch(S)\ess$, choose $X ∈ \Sch(S)$
  and any filtred inverse system $\{U_i\}_{i∈ I}$ of open subschemes with affine
  transition maps with intersection $X' \isom ∩_{i ∈ I} U_i$, and write
  $$
  \Fh\ess(X') := \varinjlim_{i ∈ I} \Fh(U_i).
  $$
  Then, $\Fh\ess$ is well-defined and functorial, in particular independent of the choice of
  $X$ and $U_i$.  Moreover,
  $$
  \Fh\ess|_{\Sch(S)}=\Fh.
  $$
\end{propdef}

The proof of \ref{defi:rs2}, given below \vpageref{pf:defi:rs2}, uses the
pro-category of schemes.  To prepare for the proof, we briefly recall the
relevant definitions.

\begin{remi}[Cofiltred category]
  A category is called \emph{cofiltred} if the following conditions hold.
  \begin{enumerate}
  \item The category has at least one object.
  \item For every pair of objects $λ$, $λ'$ there is a third object $λ''$ with
    morphisms $λ'' {→} λ$, $λ'' {→} λ'$ towards both of them.
  \item For every pair of parallel morphisms $λ' \rightrightarrows λ$ there is a
    third morphism $λ'' {→} λ'$ such that the two compositions are equal.
  \end{enumerate}
\end{remi}

\begin{remi}[Pro-category of schemes]
  The objects of the pro-category of schemes are pairs $(Λ, X)$ with $Λ$ a
  small, cofiltred category and $Λ \stackrel{X}{→} \Sch(S)$ a contravariant
  functor, and the set of morphisms between two pro-schemes $(Λ, X), (M, Y)$ is
  \emph{defined} as:
  $$
  \hom_{\Pro(\Sch(S))} \bigl( (Λ, X),\, (M, Y) \bigr) := \varprojlim_{μ ∈ M}
  \varinjlim_{λ ∈ Λ} \hom_{\Sch(S)}(X_{λ}, Y_{μ}).
  $$
\end{remi}

\begin{proof}[Proof of well-definedness in Proposition~\ref{defi:rs2}]\label{pf:defi:rs2}
  We have seen in Lemma~\ref{lem:mor_ft} and Remark~\ref{rem:lem:mor_ft} that
  $\Sch(S)\ess$ is equivalent to a subcategory of $\Pro \bigl( \Sch(S) \bigr)$.
  To prove well-definedness and functoriality, it will therefore suffice to note that $\Fh$ has a
  well-defined functorial extension, say $\Fh^\Pro$ to $\Pro \bigl( \Sch(S) \bigr)$, and
  that $\Fh\ess$ is just $\Fh^{\Pro}|_{\Sch(S)\ess}$, considering $\Sch(S)\ess$
  as a subcategory of $\Pro \bigl( \Sch(S) \bigr)$ by abuse of notation.
  Indeed, for any pro-scheme $(Λ, X)$, we can define
  $$
  \Fh^{\Pro}(Λ, X) = \varinjlim_{λ ∈ Λ} \Fh(X_{λ}).
  $$
  It follows from functoriality, that given any two pro-schemes $(Λ, X)$, $(M,
  Y)$, an element of $\hom_{\Pro(\Sch(S))} \bigl( (Λ, X),\, (M, Y) \bigr)$
  induces an element of
  $$
  \varprojlim_{μ ∈ M}\varinjlim_{λ ∈ Λ}\Hom \bigl(\Fh(Y_{μ}),\, \Fh(X_{λ})
  \bigr),
  $$
  and from therefore an element in
  $$
  \Hom \Bigl( \varinjlim_{μ ∈ M} \Fh(Y_{μ}),\, \varinjlim_{λ ∈ Λ} \Fh(X_{λ})
  \Bigr).
  $$
  One checks that all this is compatible with the composition in the various
  categories.
\end{proof}

\begin{rema}\label{rem:dfg2}
  In the setting of Definition~\ref{defi:rs2}, if $X ∈ \Sch(S)$, if $x ∈ X$ and
  if $\{U_i \}_{i ∈ I}$ are the open affine subschemes of $X$ containing $x$,
  then $\Fh\ess(∩_{i ∈ I} U_i)$ is just the usual (Zariski) stalk of $\Fh$ at
  $x$.
\end{rema}

\begin{rema}\label{rem:dfgd}
  Note, however, that in general, a presheaf $\mathcal G$ on $\Sch(S)\ess$ will
  not necessarily satisfy ${\mathcal G}(∩_{i ∈ I} U_i) = \varinjlim_{i ∈ I}
  {\mathcal G}(U_i)$, but many sheaves of interest do.  One prominent example of
  a presheaf which does satisfy ${\mathcal G} = {\mathcal G}\ess$ is $Ω^n$, for
  all $n ≥ 0$.
\end{rema}

\subsection{Description of the extension functor}
\label{sec:44}
\approvals{
  Annette & yes \\
  Shane & yes\\
  Stefan & yes
}

The following is the main result of this section.  It asserts that, under good
assumptions, the values of an extended sheaf can be reconstructed from is values
on elements of $\Cur$.

\begin{propo}[Reconstruction of unramified presheaves]\label{prop:altDefrs}
  Let $\Fh$ be an unramified presheaf on $\Sch(S)$.  If $X ∈ \Sch(S)$ is any
  scheme, then
  $$
  \Fh_{\rs}(X) = \varprojlim_{W ∈ \Cur(X)} \Fh\ess(W).
  $$
\end{propo}

The proof of this Proposition~\ref{prop:altDefrs} will take the rest of the
present Section~\ref{sec:44}.  Before beginning the proof
\vpageref{proof:prop:altDefrs}, we note two lemmas that (help to) prove the
proposition in special cases.

\begin{lemm}[{cf.~\cite[Remark~1.4]{Mor12}}]\label{lemm:limitCodimOne}
  In the setting of Proposition~\ref{prop:altDefrs}, if $X ∈ \Reg(S)$, then
  $$
  \Fh(X) = \varprojlim_{x ∈ X^{(≤ 1)}} \Fh\ess \bigl(\Spec\, \sO_{X, x} \bigr)
  $$
  where $X^{(≤ 1)}$ is the subcategory of $\Cur(X)$ consisting of inclusions of
  localisations of $X$ at points of codimension $≤ 1$.
\end{lemm}
\begin{proof}
  To keep notation short, we abuse notation slightly and write $\Fh$ for the
  presheaf, as well as for its extension $\Fh\ess$ to $\Sch(S)\ess$.  Since
  $\Fh\ess(Y) = \Fh(Y)$ for all $Y ∈ \Sch(S)$, no confusion is likely to
  occur.  Using the Axiom~\ref{unr0}, we can restrict ourselves to the case
  where $X$ is a connected, regular scheme.  Given $X$, we aim to show that the
  following canonical map is an isomorphism,
  $$
  \Fh(X) → \varprojlim_{x ∈ X^{(≤ 1)}} \Fh \bigl(\Spec\, \sO_{X, x} \bigr).
  $$

  \subsubsection*{Injectivity}

  Axiom~\ref{unr1} implies that $\Fh(X) → \Fh(η)$ is injective, where $η$ is the
  generic point of $X$.  Since this factors as
  $$
  \Fh(X) → \varprojlim_{x ∈ X^{(≤ 1)}} \Fh \bigl(\Spec\, \sO_{X, x} \bigr)
  \xrightarrow{\text{canon.~projection}} \Fh(η),
  $$
  we obtain that the first map is injective.

  \subsubsection*{Surjectivity}

  Assume we are given a section
  $$
  (s_x)_{x ∈ X^{(≤ 1)}} \quad \text{of} \quad \varprojlim_{x ∈ X^{(≤ 1)}} \Fh
  \bigl( \Spec\, \sO_{X, x} \bigr).
  $$
  By definition of the groups $\Fh \bigl( \Spec\, \sO_{X, x} \bigr)$, for every
  $x ∈ X^{(≤ 1)}$ there is some open $U_x ⊂ X$ containing $x$ and an element
  $t_x ∈ \Fh(U_x)$ which represents $s_x$.  Furthermore, by the required
  coherency, for every $x ∈ X$ of codimension one there is an open subscheme
  $U_{xη}$ of $U_x∩ U_{η}$ such that the restrictions of $t_x$ and $t_{η}$ to
  $U_{xη}$ agree.  By \ref{unr1}, this means that we actually have
  $$
  t_x|_{U_x ∩ U_y} = t_y|_{U_x ∩ U_y} \quad \text{for each } x, y ∈ X \text{ of codimension one}.
  $$
  Since $\Fh$ is a Zariski sheaf by \ref{unr2}, the sections $t_x$ therefore
  lift to a section $t$ on $∪_{x ∈ X^{(1)}} U_x$, but by \ref{unr2}, we have
  $\Fh \bigl(∪_{x ∈ X^{(1)}} U_x\bigr) = \Fh(X)$.  So the map is surjective.
\end{proof}

\begin{lemm}[{cf.~\cite[Proof of 3.6.12]{Kel12}}]\label{lemm:regcase}
  In the setting of Proposition~\ref{prop:altDefrs}, if $X$ is connected,
  regular and Noetherian with generic point $η$, then the projection map
  \begin{equation}\label{equa:regCurLimProjectToGen}
    \varprojlim_{W ∈ \Cur(X)} \Fh\ess(W) → \Fh\ess(η)
  \end{equation}
  is injective.  Consequently, Proposition~\ref{prop:altDefrs} is true when $X$
  is regular.
\end{lemm}
\begin{proof}
  As before, write $\Fh$ as a shorthand for $\Fh\ess$.  Recall from
  Remark~\ref{rema:firstRemarks:regularSame} that when $X$ is regular, $\Fh(X) =
  \Fh_{\rs}(X)$.  The composition
  $$
  \Fh(X) → \varprojlim_{W ∈ \Cur(X)} \Fh(W) → \varprojlim_{x ∈ X^{(≤ 1)}} \Fh \bigl( \Spec\, \sO_{X, x} \bigr)
  $$
  is the map we have shown is an isomorphism in Lemma~\ref{lemm:limitCodimOne}.
  So to show that the first map is an isomorphism, it suffices to show that the
  second map is a monomorphism.  Since \eqref{equa:regCurLimProjectToGen}
  factors through this, it suffices to show that
  \eqref{equa:regCurLimProjectToGen} is injective.

  Assume we are given two sections
  $$
  (s_W)_{W ∈ \Cur(X)}, \: (t_W)_{W ∈ \Cur(X)} \quad\text{of }
  \varprojlim_{W ∈ \Cur(X)} \Fh(W)
  $$
  such that $s_{η} = t_{η}$.  We wish to show that $t_W = s_W$ for all $W ∈
  \Cur(X)$.

  First, we consider $W$ of the form a point $x ∈ X$.  This is by induction on
  the codimension.  We already know $s_{η} = t_{η}$ for the point $η$ of
  codimension zero, so suppose it is true for points of codimension at most
  $n-1$ and let $x$ be of codimension $n$.  By Lemma~\ref{lem:dominateByR} there
  is a point $y$ of codimension at most $n-1$ and a discrete valuation ring $R$
  satisfying the diagram
  $$
  \xymatrix@R=0pt{
    y \ar[dr] & \\
    & \Spec\, R \ar[r] & X \\
    x \ar[ur]
  }
  $$
  By the inductive hypothesis $s_y = t_y$ and then by \ref{unr1} we have
  $s_{\Spec\, R} = t_{\Spec\, R}$.  Therefore $s_{\Spec\, R}|_x = t_{\Spec\,
    R}|_x$ but by the coherency requirement on $s$ and $t$ and we have
  $s_{\Spec\, R}|_x = s_x$ and $t_{\Spec\, R}|_x = t_x$ and therefore $t_x =
  s_x$.
  
  Now for an arbitrary $W → X$ in $\Cur(X)$ with $W$ of dimension zero, if $x$
  is the image of $W$ we have $t_W = t_x|_W = s_x|_W = s_W$.  For an arbitrary
  $W$ of dimension one and generic point $\widetilde{η}$ we have
  $t_W|_{\widetilde{η}} = t_{\widetilde{η}} = s_{\widetilde{η}} =
  s_Y|_{\widetilde{η}}$ and so by \ref{unr1} we have $t_W = s_W$.
\end{proof}

\addtocounter{thm}{1}\setcounter{equation}{0}
\begin{proof}[Proof of Proposition~\ref{prop:altDefrs}]\label{proof:prop:altDefrs}
  Again we write $\Fh$ as a shorthand for $\Fh\ess$.  By
  Lemma~\ref{lemm:regcase} we have
  $$
  \Fh_\rs(X) = \varprojlim_{Y ∈ \Reg(X)} \Fh(Y) = \varprojlim_{Y ∈ \Reg(X)}
  \varprojlim_{W ∈ \Cur(Y)} \Fh(W).
  $$
  Investigating this last double limit carefully, we observe that it can be
  described as sequences $(s_{W → Y → X})$ indexed by pairs of composable
  morphisms $W → Y → X$ with $W ∈ \Cur(X)$ and $Y ∈ \Reg(X)$, and subject to the
  following two conditions:
  \begin{enumerate}
  \item For each $W' → W → Y → X$ with $W', W ∈ \Cur(X)$ and $Y ∈ \Reg(X)$ we
    have $s_{W' → Y → X} = s_{W → Y → X}$.

  \item For each $W → Y' → Y → X$ with $W ∈ \Cur(X)$ and $Y, Y' ∈ \Reg(X)$ we
    have $s_{W → Y → X} = s_{W → Y' → X}$.
  \end{enumerate}
  By forgetting $Y$ we get a natural map
  $$
  α:\varprojlim_{W→ X}\Fh(W)→ \varprojlim_{W→ Y→ X}\Fh(W)
  $$
  with $W$, $Y$ as before.  It is given by sending an element
  $$
  (t_{W → X})_{\Cur(X)} ∈ \varprojlim_{W ∈ \Cur(X)} \Fh(W)
  $$
  to the element
  $$
  s ∈ \varprojlim_{W→ Y→ X}\Fh(W) \quad \text{with $s_{W → Y → X} = t_{W → X}$.}
  $$

  In order to check that $α$ is surjective, we only need to check that for
  any
  $$
  s ∈ \varprojlim_{W→ Y→ X}\Fh(W),
  $$
  the section $s_{W→ Y→ X}$ is independent of $Y$.  Let $W→ Y_1→ X$ and $W→ Y_2→
  X$ be in the index system.  By Example~\ref{ex:cofinal}, $W$ is the
  intersection of its affine open neighbourhoods in some $Y_3 ∈
  \Sch(S)$.  Since the regular points of a scheme in $\Sch(S)$ form an open set,
  we can assume that $Y_3 ∈ \Reg(S)$.  By the description of morphisms of
  schemes essentially of finite type in Lemma~\ref{lem:mor_ft} there exists an
  open affine subscheme $U$ of $Y_3$ containing $W$ such that we can lift the
  morphisms $W → Y_1, Y_2$ to morphisms $U → Y_1, Y_2$.  Hence
  $$
  s_{W→ Y_1→ X}=s_{W→ U→ X}=s_{W→ Y_2→ X}.
  $$

  To show that $α$ is a monomorphism, it suffices to notice that each $W ∈
  \Cur(X)$ can be ``thickened'' to a $Y ∈ \Reg(X)$: By Example~\ref{ex:cofinal}
  is $W$ is of the form $\bigcap U$ for $U$ running through the open
  neighbourhoods of $W$ in some $Y' ∈ \Sch(X)$.  Let $Y$ be the open subscheme
  of regular points of $Y'$.  We claim that $W→ Y'$ factors via $Y$.  The
  generic point of $W$ maps to a generic point of $Y'$, hence $Y$ is non-empty.
  If $W$ is of dimension $0$, we are done.  If $W$ is of dimension one, consider
  the image $y'∈ Y'$ of the closed point of $W$.  Then $W=\Spec\Oh_{Y',y'}$ and
  $y'$ is a regular point because $W$ is regular.  In particular, $W$ is
  contained in $Y$.  Hence, the functor of indexing categories which sends $W→
  Y→ X$ to $W → X$ is essentially surjective, and therefore the induced
  morphism $α$ of limits is injective.
 \end{proof}

\subsection{Descent properties}
\approvals{
  Annette & yes \\
  Shane & yes\\
  Stefan & yes
}

Recall the notion of a $\cdp$-cover from Definition~\ref{defi:cdp} and the
$\rh$, $\cdh$ and $\eh$-topologies generated by $\cdp$-covers together with open
covers, resp.\ Nisnevich covers, resp.\ étale covers.

\begin{propo}\label{lemm:isrh}
  Let $S$ be a Noetherian scheme.  If $\Fh$ is an unramified presheaf on
  $\Sch(S)$ then $\Fh_{\rs}$ is an $\rh$-sheaf.  In particular, if $\Fh$ is an
  unramified Nisnevich (resp.~étale) sheaf on $\Sch(S)$ then $\Fh_{\rs}$ is a
  $\cdh$-sheaf (resp.~$\eh$-sheaf).
\end{propo}

\begin{exam}\label{exam:ocdhs}
  Let $S=\Spec\, k$ with $k$ perfect.  Then $Ω_{\rs}^n$ is an $\eh$-sheaf for
  all~$n$.  In particular, $Ω_{\rs}^n$ is a $\cdh$-sheaf since this is a weaker
  topology.
\end{exam}

\begin{proof}[Proof of Proposition \ref{lemm:isrh}]
  By Lemma~\ref{lemm:issheaf}, $\Fh_{\rs}$ is a Zariski-sheaf (resp.\ étale or
  Nisnevich sheaf).  It remains to establish the sheaf property for
  $\cdp$-morphisms.  As before, write $\Fh$ as a shorthand for $\Fh\ess$, in
  order to simplify notation.
  
  Using the description $\Fh_{\rs}(X) = \varprojlim_{W ∈ \Cur(X)} \Fh(W)$ from
  Proposition~\ref{prop:altDefrs}, in order to show that $\Fh_{\rs}$ is
  separated for $\cdp$-morphisms, it suffices to show that for every
  $\cdp$-morphism $X' → X$, every $W → X$ in $\Cur(X)$ factors through $X' → X$.
  Since $X' → X$ is completely decomposed it is true for $W$ of dimension zero,
  and therefore also true for the generic points of those $W$ of dimension one.
  To factor all of a $W$ of dimension one, we now use the valuative criterion
  for properness.
  
  Now consider a $\cdp$-morphism $X' → X$ and a cocycle
  $$
  s = (s_W)_{W ∈ \Cur(X')} ∈ \ker \Bigl( \Fh_{\rs}(X') → \Fh_{\rs}(X' ⨯_X X') \Bigr).
  $$
  We have just observed that every scheme in $\Cur(X)$ factors through $X'$
  and so making a choice of factorisation for each $W → X$ in $\Cur(X)$, and
  taking $t_W$ to be the $s_W$ of this factorisation, we have a potential
  section $t = (t_W)_{W ∈ \Cur(X)} ∈ \Fh_{\rs}(X)$, which potentially maps to
  $s$.
  
  Independence of the choice: Suppose that $f_0, f_1: W \rightrightarrows X' →
  X$ are two factorisations of some $W ∈ \Cur(X)$.  Then there is a unique
  morphism $W → X' ⨯_X X'$ such that composition with the two projections
  recovers the two factorisations.  Saying that $s$ is a cocycle is to say
  precisely that in this situation, the two $s_W$ corresponding to $f_0$ and
  $f_1$ are equal as elements of $\Fh(W)$.
  
  This independence of the choice implies that $t$ is actually a section of
  $\Fh_{\rs}(X)$.  In other words, that $s_{W}|_{W'} = s_{W'}$ for any morphism
  $W' → W$ in $\Cur(X)$.  It also implies that $t$ is mapped to $s$.
\end{proof}

%
%
\svnid{$Id: 05.tex 246 2015-02-23 09:33:16Z kebekus $}


\section{$\cdh$-differentials}
\subversionInfo
\label{section:rs-diff}
\approvals{
  Annette & yes\\
  Shane & yes\\
  Stefan & yes
}

Throughout the present Section~\ref{section:rs-diff}, the letter $k$ will always
denote a perfect field.  To keep our statements self-contained, we will repeat
this assumption at times.

\begin{rema}[Perfect fields]\label{rem:perfect}
  Perfect fields have a number of characterisations.  A pertinent one for us is:
  the field $k$ is perfect if and only if $Ω¹_{L / k} = 0$ for every algebraic
  extension $L / k$.
\end{rema}

\begin{defi}[Dvr differentials]\label{defi:rs-differentials}
  Let $k$ be a perfect field and $n ∈ \bN$ be any number.  Let $Ω^n_{\rs}$ be
  the extension of the presheaf $Ω^n$ on $\Reg(k)$ to $\Sch(k)$ in the sense of
  Definition~\ref{defi:rs1}.  A section of $Ω^n_{\rs}$ is called a \emph{dvr
    differential}.  The justification for this name is
  Proposition~\ref{prop:altDefrs}.
\end{defi}

\begin{obse}\label{obs:oisehs}
  Let $k$ be perfect field and $n ∈ \bN$ be any number.  Then $Ω^n_{\rs}$ is an
  $\eh$-sheaf.  If $X$ is regular, then $Ω^n_{\rs}(X)=Ω^n(X)$.
\end{obse}
\begin{proof}
  Recall from Example~\ref{ex:unramified} that $Ω^n$ is an unramified presheaf.
  The observations are therefore special cases of Proposition~\ref{lemm:isrh}
  and Remark~\ref{rema:firstRemarks:regularSame}, respectively.
\end{proof}

\begin{rema}
  If the characteristic of $k$ is zero, the description as dvr differentials is
  one of the equivalent characterisations of $\h$-differentials given in
  \cite[Theorem~1]{HJ}, without the terminology being introduced.
\end{rema}

\subsection{$\cdh$-differentials}
\approvals{
  Annette & yes \\
  Shane & yes\\
  Stefan & yes
}

We introduce an alternative candidate for a good theory of differentials and
show that it agrees with $Ω^n_{\rs}$, if we assume that weak resolutions of
singularities exist.

\begin{defi}[$\cdh$-differentials]\label{def:cdhdiffs}
  Let $k$ be a perfect field and $n ∈ \bN$ be any number.  Let $Ω^n_{\cdh}$ be
  the sheafification of $Ω^n$ on $\Sch(k)$ in the $\cdh$-topology.  Sections of
  $Ω^n_{\cdh}$ are called $\cdh$-differentials.  Analogously, we define
  $Ω^n_{\rh}$ and $Ω^n_{\eh}$ via the $\rh$- or $\eh$-topology.
\end{defi}

\begin{rema}[Sheafification in characteristic zero]
  If $k$ has characteristic zero, then $Ω^n_{\cdh}$, $Ω^n_{\rh}$ and $Ω^n_{\eh}$
  agree.  In fact, $Ω^n_{\cdh}$ is even an $\h$-sheaf, \cite[Theorem~3.6]{HJ}.
\end{rema}

\begin{rema}[Comparison map]
  In the setting of Definition~\ref{def:cdhdiffs}, recall from
  Observation~\ref{obs:oisehs} that $Ω^n_{\rs}$ is an $\eh$-sheaf.  By the
  universal property of sheafification, there are canonical morphisms
  $$
  Ω^n_{\rh} → Ω^n_{\cdh} → Ω^n_{\eh} → Ω^n_{\rs}.
  $$
\end{rema}

We aim to compare these sheaves.  As it will turn out, the comparison is a
question about torsion.  The decisive input is the following theorem:

\begin{thm}[Killing torsion in $Ω$]\label{hypH}
  Let $k$ be a perfect field, $Y∈\Sch(k)$ integral, $n∈ \bN$ and
  \mbox{$ω∈\torΩ^n(Y)$}.  Then there there is a birational proper morphism $π:
  \wtilde{Y} → Y$ such that the image of $ω$ in $Ω^n(\wtilde{Y})$ vanishes.
\end{thm}

\begin{rema}
  The above result is a consequence of weak resolution of singularities, but in
  fact weaker.  We will give a direct proof in Appendix~\ref{sec:apphypH}.
\end{rema}

\begin{coro}\label{hypHCover}
  Let $k$ be a perfect field, $Y ∈ \Sch(k)$, $n ∈ \bN$ and $ω ∈ Ω^n(Y)$ an
  element such that $ω|_{y} = 0$ for every point $y ∈ Y$.  Then there exists a
  $\cdp$-morphism (Definition~\ref{defi:cdp}) $Y' → Y$ such that $ω|_{Y'} = 0$.
\end{coro}
\begin{proof}
  The proof is by induction on the dimension of $Y$.  If the dimension of $Y$ is
  zero then $Y_{\red} → Y$ is a $\cdp$-morphism such that $ω|_{Y_{\red}}$
  vanishes, so suppose that $Y$ is of dimension $n > 0$ and that the result is
  true for schemes of dimension $< n$.  Replacing $Y$ by its reduced irreducible
  components we can assume that $Y$ is integral.  Then Theorem~\ref{hypH} gives
  a proper birational morphism $Y' → Y$ for which $ω|_{Y'}$ vanishes.  Let $Z ⊂
  Y$ be a closed nowhere dense subscheme outside of which $Y' → Y$ is an
  isomorphism, and $Z' → Z$ a $\cdp$-morphism provided by the inductive
  hypothesis.  Then $Z' \amalg Y' → Y$ is a $\cdp$-morphism on which $ω$
  vanishes.
\end{proof}

\begin{thm}\label{prop:cdhDescentModTorsion}
  Let $k$ be a perfect field, $n ∈ \bN$ be any number and $X ∈ \Reg(k)$ any
  regular scheme.  Then the following canonical morphisms are all isomorphisms:
  \begin{equation}\label{eq:hjlb}
    Ω^n(X) \stackrel{\sim}{→} Ω^n_{\rh}(X) \stackrel{\sim}{→} Ω^n_{\cdh}(X) \stackrel{\sim}{→} Ω^n_{\eh}(X) \stackrel{\sim}{→} Ω^n_{\rs}(X).
  \end{equation}
\end{thm}

The isomorphism $Ω^n(X) \cong Ω^n_{\eh}(X)$ has been shown previously by Geisser
in \cite[Theorem~4.7]{Gei06} assuming the existence \emph{strong} resolutions of
singularities exists, i.e., under the additional assumption that any birational
proper morphism between smooth varieties can be refined by a series of blow-ups
with smooth centres.

\begin{proof}[Proof of Theorem~\ref{prop:cdhDescentModTorsion}]
  We formulate the proof in the case of the $\cdh$-topology.  The same arguments
  also apply, \emph{mutatis mutandis}, in the other cases.

  \subsubsection*{Step 1, proof of~\eqref{eq:hjlb} up to torsion}

  Since $X$ is regular, the composition of natural maps,
  \begin{equation}\label{eq:sfsg}
    Ω^n(X) → Ω^n_{\cdh}(X) → \underbrace{Ω^n_{\rs}(X)}_{\mathclap{= Ω^n(X) \text{
          by Rem.~\ref{rema:firstRemarks:regularSame}}}},
  \end{equation}
  is an isomorphism by Remark~\ref{rema:firstRemarks:regularSame}.  We claim
  that the direct complement $T_{\cdh}(X)$ of $Ω^n(X)$ in $Ω^n_{\cdh}(X)$ is torsion.  To
  this end, we may assume $X$ is connected, say with function field $K$.
  Applying the functor $(\bullet)\ess$ to Sequence~\eqref{eq:sfsg}, we obtain
  \begin{equation}\label{equa:omegacdhregIsoGen}
    \underbrace{(Ω^n)\ess(K)}_{\mathclap{= Ω^n(K) \text{ by Rem.~\ref{rem:dfg2}}}}
    \longrightarrow
    (Ω^n_{\cdh})\ess(K)
    \longrightarrow
    \underbrace{(Ω^n_{\rs})\ess(K)}_{\mathclap{= Ω^n(K) \text{ by
          Rems.~\ref{rema:firstRemarks:regularSame}, \ref{rem:dfg2}}}}.
  \end{equation}

  We claim that both morphisms in~\eqref{equa:omegacdhregIsoGen} are
  isomorphisms.  Equation~\eqref{eq:hjlb} then follows from \ref{unr1}.  In
  order to justify the claim, it suffices to show that the first morphism is
  surjective.  Suppose that $s ∈ (Ω^n_{\cdh})\ess(K)$.  By the definition of
  $(-)\ess$, there is some open affine $U ⊆ X$ such that $s$ is represented by a
  section $s' ∈ (Ω^n_{\cdh})(U)$.  Then, there exists a $\cdh$-cover $V → U$
  such that $s'|_V$ is in the image $Ω^n(V) → Ω^n_\cdh(V)$.  Replacing $U$ by
  some smaller open affine, we can assume that $V = U$.  That is, the
  representative $s' ∈ (Ω^n_{\cdh})(U)$ is in the image of $(Ω^n)(U) →
  (Ω^n_{\cdh})(U)$.  Hence, the section $s ∈ (Ω^n_{\cdh})\ess(K)$ it represents
  is in the image of $(Ω^n)\ess(K) → (Ω^n_{\cdh})\ess(K)$.

  \subsubsection*{Step 2, vanishing of torsion}
  
  We will
  now show that $T_{\cdh}(X) = 0$.  Let $ω ∈ T_{\cdh}(X)$ be any element---so
  $ω$ is an element of $Ω^n_{\cdh}(X)$ which vanishes in $(Ω^n_{\cdh})\ess(K)$.
  We aim to show that $ω = 0$.  To this end, we construct a number of covering
  spaces, associated groups and preimages, which lead to
  Diagrams~\eqref{eq:pfff1} and \eqref{eq:pfff2} below.  We may assume without
  loss of generality that $X$ is reduced.

  By definition of sheafification, there exists a $\cdh$-cover $V → X$ of $X$
  for which $ω|_V$ is in the image of $Ω^n(V) → Ω^n_{\cdh}(V)$.  Again we may
  assume that $V$ is reduced.  We choose one element $ω' ∈ Ω^n(V)$ contained in
  the preimage of $ω$.  As $V → X$ is a $\cdh$-cover, we have a factorisation
  $\Spec\, K → V → X$.  Recalling from Step~1 that $(Ω^n)\ess(K) →
  (Ω^n_{\cdh})\ess(K)$ is an isomorphism, it follows that $ω' ∈ Ω^n(V)$ is a
  torsion element.

  By Corollary~\ref{cor:normalform} we may refine the covering map and assume
  without loss of generality that it factorises as follows,
  $$
  V \xrightarrow{\text{Nisnevich cover\vphantom{p}}} Y \xrightarrow{\text{proper
      birational}} X.
  $$
  Since $ω' ∈ Ω(V)$ is torsion, Hypothesis~\ref{hypH} gives a proper birational
  morphism $V' → V$ such that $ω'|_{V'} = 0$.  Finally,
  Lemma~\ref{lem:makecover} allows one to find $Y''$, $Y'$ fitting into the
  following commutative diagram
  \begin{equation}\label{eq:pfff1}
    \begin{gathered}
      \xymatrix{
        Y'' \ar[rrrr]^{\text{Nisnevich cover}} \ar[d]_{\text{proper biratl.}} &&&& Y' \ar[d]^{\text{proper biratl.}} & \\
        V' \ar[rr]_{\text{proper biratl.}} && V \ar[rr]_{\text{Nisnevich cover}} && Y \ar[rr]_{\text{proper biratl.}} && X
      }
    \end{gathered}
  \end{equation}
  Proposition~\ref{prop:cdpRefinable} implies that $Y' → X$ and $Y'' → Y$ are
  $\cdh$-covers.  A diagram chase will now finish the argument:
  \begin{equation}\label{eq:pfff2}
    \begin{gathered}
      \xymatrix@!=0.3cm{ %
        & \underset{\text{contains }ω|_{Y''} = 0}{Ω^n_{\cdh}(Y'')}
        &&&& Ω^n_{\cdh}(Y') \ar[llll] \\
        \underset{\text{contains }ω'|_{Y''} = 0}{Ω^n(Y'')} \ar[ur] \\
        & Ω^n_{\cdh}(V') \ar[uu]
        && \underset{\text{contains }ω|_V}{Ω^n_{\cdh}(V)} \ar[ll]
        && Ω^n_{\cdh}(Y) \ar[ll] \ar[uu]
        && \underset{\text{contains }ω}{Ω^n_{\cdh}(X)} \ar[ll] \\
        \underset{\text{contains }ω'|_{V'} = 0}{Ω^n(V')} \ar[uu] \ar[ur]
        && \underset{\text{contains }ω'}{Ω^n(V)} \ar[ll] \ar[ur]
      }
    \end{gathered}
  \end{equation}
\end{proof}

\begin{prop}\label{prop:cdhInRs}
  Let $k$ be a perfect field and $n ∈ \bN$ be any number.  The canonical map
  $Ω^n_{\cdh} → Ω^n_{\rs}$ is a monomorphism.
\end{prop}
\begin{proof}
  Let $X ∈ \Sch(k)$ and $ω$ be in the kernel of $Ω^n_{\cdh}(X) → Ω_{\rs}^n(X)$.
  Choose a $\cdh$-cover $X' → X$ such that the restriction of $ω$ to $X'$ is in
  the image of $Ω^n(X') → Ω^n_{\cdh}(X')$, so now we have a section $ω'$ in the
  kernel of $Ω^n(X') → Ω^n_{\rs}(X')$ and we wish to show that it vanishes on a
  $\cdh$-cover of $X'$.  By Corollary~\ref{hypHCover} it suffices to show that
  $ω'$ vanishes on every point of $X'$.  Let $x ∈ X'$ be a point,
  $\overline{\{x\}}$ its closure in $X'$ with the reduced scheme structure, and
  let $V = (\overline{\{x\}})_{\mathrm{reg}}$ be its regular locus.  Since the
  section $ω'$ vanishes in $Ω^n_{\rs}(X')$, it vanishes on every scheme in
  $\Reg(X')$, and in particular, on $V$.  But the generic point of $V$ is
  isomorphic to $x$, and therefore $ω'$ vanishes on $x$.
\end{proof}

\begin{propo}\label{prop:cdhRsAgree}
  Let $k$ be a perfect field and $n ∈ \bN$ be any number.  Assume that weak
  resolutions of singularities exist for schemes defined over $k$.  Then the
  natural maps
  $$
  Ω^n_{\rh}→ Ω^n_{\cdh}→Ω^n_{\eh}→ Ω^n_{\rs}
  $$
  of presheaves on $\Sch(k)$ are isomorphisms.  As such, all four are
  unramified.
\end{propo}
\begin{proof}
  We argue by induction on the dimension.  In dimension zero, there is nothing
  to show.  Now, let $X ∈ \Sch(k)$ be of positive dimension.  We may assume that
  $X$ is reduced and hence generically regular.  Let $\widetilde{X} → X$ be a
  desingularisation, isomorphic outside a proper closed set $Z \subsetneq X$ and
  with exceptional locus $E \subsetneq \widetilde{X}$.  It follows from
  $\cdp$-descent that the sequences
  $$
  0 → Ω^n_{\bullet}(X) → Ω^n_{\bullet} \bigl( \widetilde{X} \bigr) \oplus
  Ω_{\bullet}^n(Z)→Ω^n_{\bullet}(E) \qquad \text{where } \bullet ∈ \{\rh, \cdh,
  \eh, \rs\}
  $$
  are all exact.  By inductive hypothesis the comparison map is an isomorphism
  for $Z$ and $E$.  It is also an isomorphism for $\widetilde{X}$, by
  Proposition~\ref{prop:cdhDescentModTorsion}.  Hence it is also an isomorphism
  for $X$.
\end{proof}

\begin{rema}
  It is currently unclear to us if Theorem~\ref{hypH} suffices to establish the
  conclusion of Proposition~\ref{prop:cdhRsAgree}.  For all we know,
  Theorem~\ref{hypH} only implies that the natural map $Ω^n_{\rs} →
  Ω^n_{\bullet,\rs}$ for $\bullet∈\{\rh,\cdh,\eh\}$ is isomorphic.
\end{rema}

\subsection{Torsion of $\cdh$-sheaves}\label{sec:torsion_cdh}
\approvals{
  Annette & yes \\
  Shane & yes\\
  Stefan & yes
}

Over fields of characteristic zero, the results \cite[Theorem~1, Remark~3.11 and
Proposition~4.2]{HJ} show that $Ω^n_{\cdh} = Ω^n_{\rs}$ is torsion free.  We are
going to show that this fails in positive characteristic.

\begin{exam}[Existence of $\cdh$-torsion]\label{coro:rs_not_torsionfree1}
  We maintain the setting and notation of
  Example~\ref{exam:pullbackTorsionNotTorsion}: the field $k$ is algebraically
  closed of characteristic two, $Y := \Spec\, \factor{k[x,y,z]}{(y²-xz²)}$ is
  the Whitney umbrella, and $X := \Spec\, k[x]$ is its singular locus.  Write
  $\widetilde{Y} := \Spec\, k[u,z] \cong \bA²$ and consider the (birational)
  desingularisation $π: \widetilde{Y} → Y$, given by
  $$
  π^\#: \factor{k[x,y,z]}{(y²-xz²)} → k[u,z], \qquad x \mapsto u²,\: y \mapsto uz,\: z \mapsto z.
  $$
  Let $E \subsetneq \widetilde{Y}$ be the exceptional locus of $π$.  In other
  words, $E$ equals the preimage of $X$ and is hence given by $z=0$.  Note that
  the morphism $E → X$ corresponds to the ring morphism $k[x] → k[u]$, $x
  \mapsto u²$, and therefore induces the zero morphism on $Ω¹$.  We compute
  $Ω¹_{\rs}$ of $Y$, $\widetilde{Y}$, $X$, and $E$.  The last three are regular,
  hence $Ω¹ → Ω¹_{\rs}$ is an isomorphism on these varieties.  Since $Ω¹_{\cdh}$
  is a $\cdh$-sheaf and $\widetilde{Y} \amalg X → Y$ is a $\cdh$-cover, we have
  the following exact sequence:
  $$
  0 → Ω¹_{\cdh}(Y) \xrightarrow{a} Ω¹\bigl( \widetilde{Y} \bigr) \oplus Ω¹(X)
  \xrightarrow{b} Ω¹(E).
  $$
  We have seen above that $Ω¹(X) ⊆ \ker b \isom Ω¹_{\cdh}(Y)$.  The associated
  sections of $Ω¹_{\cdh}(Y)$ vanish on $Y \setminus X$ and are therefore torsion
  on $Y$.  It follows that $\tor Ω_{\rs}¹(Y) \ne 0$, and that there are torsion
  elements whose restrictions to $X$ do not vanish, and furthermore, are not
  torsion elements of $Ω¹_{\cdh}(X)$.
\end{exam}

\begin{coro}\label{coro:rs_not_torsionfree2}
  For perfect fields of positive characteristic, the sheaves
  $Ω^{\bullet}_{\cdh}$ are not torsion-free in general.  \qed
\end{coro}

\begin{coro}\label{coro:torsionIsNotAPresheafForcdh}
  For perfect fields of positive characteristic, the pull-back maps of
  $Ω¹_{\cdh}(·)$ do not induce pull-back maps between the groups $\tor
  Ω¹_{\cdh}(·)$.  In other words, $\tor Ω¹_{\cdh}$ is generally not a presheaf
  on $\Sch(k)$.  \qed
\end{coro}

\begin{rema}
  The same computation works for any extension of $Ω¹$ to a sheaf on $\Sch$
  which has $\cdp$-descent and agrees with Kähler differentials on regular
  varieties.  By Theorem \ref{hypH} this includes $\reg$-, $\rh$-, and
  $\eh$-differentials.
\end{rema}

%
%
\svnid{$Id: 06.tex 239 2015-02-20 14:28:18Z kebekus $}

\section{Separably decomposed topologies}
\subversionInfo
\label{section:newTopologies}
\approvals{
  Annette & yes \\
  Shane & yes \\
  Stefan & yes
}

In many applications, de Jong's theorem on alterations \cite{deJ96} and Gabber's
refinement \cite{ILO} have proved a very good replacement for weak resolution of
singularities, which is not (yet) available in positive characteristic.  It is
natural to ask if one can pass from the $\eh$-topology to a suitable refinement
that allows alterations as covers, but still preserves the notion of a
differential in the smooth case.  This turns out impossible.

\subsection{h-topology}
\approvals{
  Annette & yes \\
  Shane & yes \\
  Stefan & yes
}

We record the following for completeness.

\begin{lemm}[Sheafification of differentials in the $\h$-topology]\label{lemm:hNoGood}
  Assume that there is a positive integer $m$ such that $m = 0$ in $\sO_S$ (for
  example, $S$ might be the spectrum of a field of positive characteristic).  If
  $n > 0$ is any number, then the $h$-sheafification $Ω_h^n$ of $Ω^n$ on
  $\Sch(S)$ is zero.  In fact, even the $h$-separated presheaf associated to
  $Ω^n$ is zero.
\end{lemm}
\begin{proof}
  Since the $\h$-topology is finer than the Zariski topology, it suffices to
  prove the statement for affine schemes.  We claim that for any ring $A$, and
  any differential of the form $da ∈ Ω \bigl(\Spec\, A \bigr)$, there exists an
  $\h$-cover $Y → \Spec\, A$ such that $da$ is sent to zero in $Ω(Y)$.  Indeed,
  it suffices to consider the finite surjective morphism
  $$
  \Spec\, \factor{A[T]}{(T^m - a)} → \Spec\, A.
  $$
  The image of $da$ under this morphism is $da = d(T^m) = m T^{m - 1} · dT =
  0$.
\end{proof}

\subsection{$\sdh$-topology}
\approvals{
  Annette & yes \\
  Shane & yes\\
  Stefan & yes
}

To avoid the phenomenon encountered in Lemma~\ref{lemm:hNoGood}, one could try
to consider the following topology, which is coarser than the the $\h$-topology.
We only allow proper maps that are generically separable.  By making the notion
stable under base change, we are led to the following notion.

\begin{defi}[$\sdh$-topology]\label{defi:sdhtopology}
  We define the $\sdh$-topology on $\Sch(S)$ as the topology generated by the
  étale topology, and by proper morphisms $f: Y → X$ such that for every $x ∈ X$
  there exists $y ∈ f^{-1}(x)$ with $[k(y): k(x)]$ finite separable.
\end{defi}

\begin{exam}\label{ex:sdhc1}
  Let $π: \widetilde{X} → X$ be a proper morphism and let $Z ⊂ X$ be a closed
  subscheme such that $π$ is finite and étale over $X \setminus Z$.  The obvious
  map $\widetilde{X} \amalg Z → X$ is then an $\sdh$-cover.
\end{exam}

\begin{rema}
  In characteristic zero, the $\h$- and $\sdh$-topologies are the same,
  \cite[Proof of Proposition~3.1.9]{Voev96}.
\end{rema}

\begin{exam}[Failure of $\sdh$-descent]\label{exam:cdhDecsentFailure}
  Let $k$ be a perfect field of characteristic $p$, let $n$ be a positive
  integer,
  $$
  S := \frac{k[x,y,z]}{(z^p+zx^n-y)} \isom k[x,z], \qquad \text{ and }\qquad R
  := k[x,y].
  $$
  The obvious morphism $R → S$ defines a covering map,
  \[
  π: \underbrace{\Spec\, S}_{=: \widetilde{X}} → \underbrace{\Spec\, R.}_{ =:
    X\vphantom{\widetilde{X}}}
  \]
  Note that both $X$ and $\widetilde{X}$ are regular.  The covering map is
  finite of degree $p$.  It is étale outside the exceptional set $Z=V(x)⊂ X$.
  Indeed, the minimal polynomial of $z$ over $k[x,y]$ is $T^p+x^nT-y$ with
  derivative $x^n$.  Hence $π$ is an alteration and generically separable.  This
  also means that $\widetilde{X} {⨯_X} \widetilde{X}$ is a regular outside of $Z
  {⨯_X} \widetilde{X} {⨯_X} \widetilde{X}$.
\end{exam}

We have observed in Example~\ref{ex:sdhc1} that $\widetilde{X} \amalg Z → X$ is
an sdh-cover.  We will now show that $\sdh$-descent fails for $Ω¹_{\rs}$ and
this cover.  Example~\ref{exam:cdhDecsentFailure} will also be used in
Lemma~\ref{lemm:nosaltdescent} to show that $Ω¹_{\rs}$ does not have descent for
a local (in the birational geometry sense) version of the $\sdh$-site.  The
reader who wishes to explore it in more detail can consult
Appendix~\ref{appendixCounterexample} where we make some explicit calculations.

\begin{propo}[Failure of $\sdh$-descent]\label{prop:noSdhDescent}
  In the setting of Example~\ref{exam:cdhDecsentFailure}, $\sdh$-descent fails
  for $Ω¹_{\cdh}$ and the cover $\widetilde{X} \amalg Z → X$.  In other words,
  $Ω¹(X)\neqΩ¹_{\sdh}(X)$ for this regular $X$.
\end{propo}

\begin{rema}
  The argument shows more: no presheaf $\Fh$ with $\Fh|_{\Reg} \cong Ω¹$ can
  ever be an $\sdh$-sheaf.
\end{rema}

\begin{proof}[Proof of Proposition~\ref{prop:noSdhDescent}]
  We have to discuss exactness (or rather failure of exactness) of the sequence
  \begin{equation}\label{eq:dflh}
    \begin{split}
      0 → Ω¹_{\rs}(X) \xrightarrow{α} Ω¹_{\rs}(\widetilde{X})\oplus Ω¹_{\rs}(Z) \hspace{5cm}\\
      \xrightarrow{β} Ω¹_{\rs}(\widetilde{X}⨯_X\widetilde{X}) \oplus
      Ω¹_{\rs}(\widetilde X ⨯_X Z ) \oplus Ω¹_{\rs}(Z ⨯_X \widetilde X ) \oplus
      Ω¹_{\rs}(Z ⨯_X Z).
    \end{split}
  \end{equation}
  Notice that $Z {⨯_X} Z = Z$, and that $\widetilde{Z} := \widetilde{X} ⨯_X Z$
  is given as
  \begin{equation}\label{eq:Zpreimage}
    \widetilde Z=\Spec\, \frac{k[x,y,z]}{(z^p+zx^n-y,x)} \isom
    \Spec\, \frac{k[y,z]}{(z^p-y)} \isom \Spec\, k[z].
  \end{equation}
  Using smoothness of $X, Z, \widetilde{X}$ and $\widetilde{Z}$,
  Sequence~\eqref{eq:dflh} simplifies to
  $$
  0 → Ω¹(X) \xrightarrow{α} Ω¹(\widetilde{X})\oplusΩ¹(Z) \xrightarrow{β}
  Ω¹_{\rs}(\widetilde{X}⨯_X\widetilde{X}) \oplus Ω¹(\widetilde Z) \oplus
  Ω¹(\widetilde Z) \oplus Ω¹(Z).
  $$
  We will work with this simplified description.  Recalling that $Ω¹(Z) =k[y]
  · dy$, consider the element in the middle that is given by
  $$
  0 \oplus dy ∈ Ω¹(\widetilde{X})
  \oplus Ω¹(Z).
  $$

  We claim that $β(0 \oplus dy) = 0$.  This will be shown by considering the
  four components of $β(0 \oplus dy)$ one at a time.  The component in
  $Ω¹_{\rs}(\widetilde{X}⨯_X\widetilde{X})$ clearly vanishes because the first
  component of $0 \oplus dy$ does.  The components in $Ω¹(\widetilde Z) \oplus
  Ω¹(\widetilde Z)$ vanish because $d(π|_{\widetilde Z}) : Ω¹(Z) →
  Ω¹(\widetilde{Z})$ is the map
  $$
  k[y] · dy → k[z] dz; \qquad f(y) · dy \mapsto f(z^p) · d(z^p) = 0.
  $$
  Finally, recall that we used the identity $Z ⨯_XZ = Z$ in the simplification.
  The two restriction maps $Ω¹(Z) → Ω¹(Z ⨯_XZ) = Ω¹(Z)$ agree, so that the last
  component vanishes as well.  The claim is thus shown.
  
  On the other hand, $0 \oplus dy$ cannot be in the image of $α$ because the
  restriction map $Ω¹(X) → Ω¹(\widetilde{X})$ is injective.  In summary, we see
  that \eqref{eq:dflh} can not possibly be exact.  This concludes the proof.
\end{proof}

\subsection{The site $\salt$}
\approvals{
  Annette & yes \\
  Shane & yes\\
  Stefan & yes
}

As the problem that arises in Example~\ref{exam:cdhDecsentFailure} seems to lie
in the non-separable locus of $X' → X$, one could try removing the need for $Z$,
by considering the following version of \cite[Exposé~II, Section~1.2]{ILO}.

\begin{defi}[Site $\salt(X)$]\label{defi:saltSite}
  Let $S$ be Noetherian, $X∈\Sch(S)$.  We define the site $\salt(X)$ as follows.
  The objects are those morphisms $f: X' → X$ in $\Sch(S)$ such that $X'$ is
  reduced, and for every generic point $\xi ∈ X'$, the point $f(\xi)$ is a
  generic point of $X$ and moreover, $k(\xi) / k \bigl( f(\xi) \bigr)$ is finite
  and separable.  The topology is generated by the étale topology, and morphisms
  of $\salt(X)$, which are proper.  Note that by virtue of being in $\salt(X)$,
  the latter are automatically generically separable.
\end{defi}

\begin{exam}
  ---
  \begin{enumerate}
  \item Let $X$ be integral and $f:X' → X$ a blowing-up.  Then $f$ is proper and
    birational, and hence an $\salt$-cover.

  \item Let $X$ be reduced with irreducible components $X_1, X_2$.  Then
    $\widetilde{X}=X_1 \amalg X_2 → X$ is an $\salt$-cover.

  \item The morphism $π : \widetilde{X} → X$ of
    Example~\ref{exam:cdhDecsentFailure} is an $\salt$-cover.
  \end{enumerate}
\end{exam}

The category $\salt(X)$ admits fibre products in the categorical sense, which
can be calculated as follows: For morphisms $Y'→ Y$ and $Y''→ Y$ in $\salt(X)$
let $Y' \stimes_Y Y''$ be the union of the reduced irreducible components of the
usual fibre product of schemes, $Y'{⨯_Y}Y''$, that dominate an irreducible
component of $X$.

\begin{exam}
  ---
  \begin{enumerate}
  \item Let $X$ be reduced and connected with irreducible components $X_1$, $X_2$.
    Let $\widetilde{X} := X_1 \amalg X_2$.  Then
    $$
    \widetilde{X} ⨯_X \widetilde{X} = X_1 \amalg X_2 \amalg (X_1 ∩ X_2) \amalg
    (X_2 ∩ X_1).
    $$
    In order to obtain the product in $\salt$ we have to drop the components which
    are not dominant over an irreducible component of $X$ and get
    $$
    \widetilde{X}\stimes_X\widetilde{X}=\widetilde{X}.
    $$

  \item Let $k$ be a perfect field, $C/k$ a nodal curve with normalisation
    $\widetilde{C}$ ---this means that locally for the étale topology we are in
    the situation of the previous example.  Then $\widetilde{C} ⨯_C
    \widetilde{C}$ has one irreducible component isomorphic to $\widetilde{C}$
    and two isolated points.  Hence
    $$
    \widetilde{C} \stimes_C \widetilde{C} \isom \widetilde{C}.
    $$

  \item In the setting of Example~\ref{exam:cdhDecsentFailure}, the canonical
    inclusion $\widetilde{X} \stimes_X \widetilde{X} → \widetilde{X} ⨯_X
    \widetilde{X}$ is an isomorphism.
  \end{enumerate}
\end{exam}

Since we have access to fibre products, a presheaf $\Fh$ on $\salt(X)$ is an
\emph{$\salt$-sheaf} if the following sequence is exact for all covers $X_2→
X_1$,
$$
0→ \Fh(X_1)→ \Fh(X_2)→ \Fh(X_2\stimes_{X_1}X_2).
$$
By de Jong's theorem on alterations \cite{deJ96}, the system of covers $Y → X$
with $Y$ regular is cofinal.

\begin{lemm}\label{lem:Xa}
  Let $k$ be perfect, $X∈\Sch(k)$.  For general $X$, the presheaf $Ω¹_{\rs}$
  restricted to $\salt(X)$ is not an $\salt$-sheaf.  In fact, it is not
  separated.
\end{lemm}
\begin{proof}
  Assume that $Ω¹_{\rs}$ is separated.  In other words, assume that the natural
  map $Ω¹_{\rs}(Y) → Ω¹_{\rs}(Y')$ is injective for all separable alterations
  $Y'→ Y$ in $\salt(X)$.  Let $X$ be irreducible.  Choose a separable alteration
  $Y → X$ with $Y$ regular.  Then, we have injective maps
  $$
  Ω¹_{\rs}(X) → \underbrace{Ω¹_{\rs}(Y)}_{=Ω¹(Y)} → Ω¹ \bigl( k(Y) \bigr).
  $$
  The composition factors via $Ω¹ \bigl( k(X) \bigr) → Ω¹ \bigl( k(Y) \bigr)$.
  This map is also injective, because $k(Y)/k(X)$ is separable.  In total, we
  obtain that the map
  $$
  Ω¹_{\rs}(X) → Ω¹ \bigl( k(X) \bigr)
  $$
  is injective.  In particular, we obtain that $Ω¹_{\rs}(X)$ is torsion-free,
  contradicting Corollary~\ref{coro:rs_not_torsionfree2}.
\end{proof}

Torsion in $Ω¹_{\rs}(X)$, which played a role in the proof of
Lemma~\ref{lem:Xa}, occurs only for singular $X$.  The following example shows,
however, that $\salt$-descent also fails for regular $X$.

\begin{lemm}\label{lemm:nosaltdescent}
  Let $k$ be perfect.  Then, $Ω¹_{\rs}$ restricted to $\salt(X)$ does not have
  $\salt$-descent for the morphism $π : \widetilde{X}→ X$ of
  Example~\ref{exam:cdhDecsentFailure} if $n≥ 2$.  In particular,
  $Ω¹_{\salt}(X)\neq Ω¹(X)$ for this particular $X$.
\end{lemm}

The reader interested in following the computations here might also wish to look
at Appendix~\ref{appendixCounterexample} first, where many of the relevant rings
and morphisms are explicitly computed.

\begin{proof}
  We consider the following commutative diagram,
  $$
  \xymatrix{%
    0 \ar[r] & Ω¹(X) \ar[d] \ar[r] & Ω¹(\widetilde{X}) \ar[d] \ar[r]^(.4){α} \ar[dr]_(.45){β} & Ω¹_{\rs}(\widetilde{X}{⨯_X}\widetilde{X}) \ar[d] \\
    0 \ar[r] & Ω¹ \bigl( k(X) \bigr) \ar[r] & Ω¹ \bigl(k(\widetilde{X}) \bigr) \ar[r] & Ω¹ \bigl( k(\widetilde{X}{⨯_X}\widetilde{X}) \bigr).
  }
  $$
  We wish to show that the top row is not exact.  As $X$ and $\widetilde{X}$ are
  smooth, the left two vertical morphisms are monomorphisms.  Moreover, since
  $\widetilde{X} → X$ is generically étale and $Ω¹$ is an étale sheaf, the lower
  row is exact.  Consequently, the top row is exact at $Ω¹(X)$.  Now, for the
  moment, we ask the reader to admit the following claim.

  \subsubsection*{Claim:} The kernels of $α$ and $β$ are equal.

  \subsubsection*{Application of the claim:} If this claim holds, then from a
  diagram chase, exactness of the top row at $Ω¹(\widetilde{X})$ would imply
  that $Ω¹(X)$ is the intersection of $Ω¹(k(X))$ and $Ω¹(\widetilde{X})$ inside
  $Ω¹(k(\widetilde{X}))$.  That is, exactness at $Ω¹(\widetilde{X})$ would imply
  the that inclusion
  \[
  \underbrace{k[x, y] · dx \oplus k[x, y] · dy}_{=Ω¹(X)} ⊆
  \underbrace{k(x, y) · dx \oplus k(x, y) · dy}_{= Ω¹\bigl( k(X) \bigr)}
  \; ∩ \; \underbrace{k[x, z] · dx \oplus k[x, z] · dz}_{= Ω¹(\widetilde
    X)}
  \]
  is in fact an equality inside $Ω¹\bigl( k(\widetilde{X}) \bigr) = k(x, z)
  · dx \oplus k(x, z) · dz$, where $y = z^p + zx^n$.

  However, the element
  \[
  x^{-1} · dy = n z x^{n - 2} · dx + x^{n-1} · dz ∈ Ω¹\bigl(
  k(\widetilde{X}) \bigr)
  \]
  is in the intersection on the right, but cannot come from an element on the
  left, since for any element coming from the left the coefficient of $dz$ is
  divisible by $x^n$.  Hence, the inclusion is strict, as long as our claim that
  $\ker α = \ker β$ holds.
  
  \subsubsection*{Proof of the Claim:} To see that $\ker α = \ker β$, first note
  that since $\widetilde{X}⨯_X\widetilde{X}$ is of dimension $≤ 2$, there is a
  strong resolution of singularities, that is, a proper birational map $ρ:
  X^{(2)} → \widetilde{X}⨯_X\widetilde{X}$, which is an isomorphism outside
  $Z⨯_X\widetilde{X} ⨯_X \widetilde{X}$.  Setting $Z^{(2)} = (Z ⨯_X
  \widetilde{X} ⨯_X \widetilde{X})_{\operatorname{red}}$, we now have a proper
  $\cdh$-cover $Z^{(2)} \amalg X^{(2)}$ of $\widetilde{X} ⨯_X \widetilde{X}$.
  This is useful because we have seen in Example~\ref{exam:ocdhs} that
  $Ω¹_{\rs}$ is a $\cdh$-sheaf and so
  $$
  Ω¹_{\rs} \bigl( \widetilde{X} ⨯_X \widetilde{X} \bigr) → Ω¹_{\rs} \bigl(
  X^{(2)} \amalg Z^{(2)} \bigr)
  $$
  is injective.  Now since $Z⨯_X\widetilde{X}⨯_X\widetilde{X} =
  (Z⨯_X\widetilde{X}) ⨯_Z (Z⨯_X\widetilde{X}) = \widetilde{Z} ⨯_Z
  \widetilde{Z}$, one calculates $Z^{(2)}$ as in \eqref{eq:Zpreimage} as
  $$
  \bigl( \Spec\, k[z] \otimes_{k[y]} k[z] \bigr)_{\operatorname{red}} = \bigl(
  \Spec\, k[z] \otimes_{k[z^p]} k[z] \bigr)_{\operatorname{red}} = \Spec\, k[z].
  $$
  From this calculation we glean two important pieces of information.  Firstly
  $Z^{(2)}$ is smooth, and so
  $$
  Ω¹_{\rs} \bigl( X^{(2)} \amalg Z^{(2)} \bigr) = Ω¹ \bigl( X^{(2)} \amalg Z^{(2)} \bigr)
  $$
  and since $Ω$ is torsion-free on regular schemes, the morphism
  $$
  Ω¹ \bigl( X^{(2)} \bigr) \oplus Ω¹ \bigl( Z^{(2)} \bigr) → Ω¹ \bigl( k(X^{(2)}) \bigr)
  \oplus Ω¹ \bigl( k(Z^{(2)}) \bigr)
  $$
  is injective.  Since $X^{(2)} → \widetilde{X}⨯_X\widetilde{X}$ was birational,
  all this implies that the morphism
  $$
  Ω¹_{\rs} \bigl( \widetilde{X} ⨯_X \widetilde{X} \bigr) → Ω¹ \bigl( k(\widetilde{X} ⨯_X \widetilde{X}) \bigr) \oplus Ω¹ \bigl( k(Z^{(2)}) \bigr)
  $$
  is injective.  So, to finish the proof of the claim, it suffices to show that
  the induced morphism $Ω¹ \bigl( \widetilde{X} \bigr) → Ω¹ \bigl( k(Z^{(2)})
  \bigr)$ is zero.  This is the second piece of information we obtain from the
  description of $Z^{(2)}$ above.  Since $Z^{(2)} = \widetilde{Z}$, the two
  compositions
  $$
  Z^{(2)} → \widetilde{X}{⨯_X}\widetilde{X} \rightrightarrows \widetilde{X}
  $$
  induced by the projections are equal, and so $Ω¹ \bigl( \widetilde{X} \bigr) →
  Ω¹ \bigl( k(Z^{(2)}) \bigr)$ is indeed zero.
\end{proof}

\begin{appendix}
%
%
\svnid{$Id: 0A.tex 246 2015-02-23 09:33:16Z kebekus $}

\section{Theorem~\ref*{hypH}}\label{appA}
\label{sec:apphypH}
\subversionInfo
\approvals{
  Annette & yes \\
  Shane & yes\\
  Stefan & yes
}

We have seen in the main text that understanding the torsion in $Ω^n$ is crucial
in order to understand $Ω^n_{\cdh}$.  The main purpose of this appendix is to
give a proof of Theorem~\ref{hypH}.  We are highly indepted to an anonymous
referee who provided the crucial reference to the result of Gabber-Rambero
\cite[Corollary 6.5.21]{GR03}, which appeared as a hypothesis in an earlier
version of this article.  We also give criteria for testing the existence of
torsion in special cases.  These considerations are independent of the main
text.

Before going into the direct proof, we explain how Theorem~\ref{hypH} follows
easily from resolution of singularities.

\begin{lemm}
  Let $k$ be perfect.  Assume weak resolution of singularities holds over~$k$,
  that is, assume that for every reduced $Y$ there is a proper birational
  morphism $X→ Y$ with $X$ smooth.  Then, Theorem~\ref{hypH} holds true.
\end{lemm}
\begin{proof}
  Let $ω ∈ \tor Ω^n(Y)$.  By definition there is a dense open subset $U⊂ Y$ such
  that $ω|_U$ vanishes.  Let $π:X→ Y$ a desingularisation.  Then $π^*ω∈ Ω^n(X)$
  is a torsion from because it vanishes on $π^{-1}U$.  As $X$ is smooth, this
  implies that $π^*ω=0$.
\end{proof}

\subsection{Valuation rings}\label{sec:val_rings}
\approvals{
  Annette & yes \\
  Shane & yes \\
  Stefan & yes
}

We give a reformulation of Theorem~\ref{hypH} in terms of vanishing of
differential forms on (not necessarily discrete) valuation rings.  In this
section, let $k$ be a perfect field.

Let $A$ be an integral $k$-algebra of finite type.  Recall that the
Riemann-Zariski space $\RZ(A)$, called the ``Riemann surface'' in \cite[§~17,
p.~110]{SZ}, as a set, is the set of (not necessarily discrete) valuation rings
of $\Frac(A)$ which contain $A$.  To a finitely generated sub-$A$-algebra $A'$
is associated the set $E(A') = \{ R ∈ \RZ(A) : A' ⊆ R \}$ and one defines a
topology on $\RZ(A)$ taking the $E(A')$ as a basis.  This topological space is
quasi-compact, in the sense that every open cover admits a finite subcover
\cite[Theorem~40]{SZ}.

Consider the following hypothesis.

\setcounter{dummy}{21}
\begin{hypoAlph}\label{hypo:valRing}
  For every finitely generated extension $K/k$ and every $k$-valuation ring $R$
  of $K$ the map $Ω^n(R) → Ω^n(K)$ is injective for all $n ≥ 0$.
\end{hypoAlph}
\begin{rema}[Hypothesis~\ref{hypo:valRing} true]
  The statement is true for $n=0$ because valuation rings are torsion free.
  \cite[Corollary 6.5.21]{GR03} states that Hypothesis~\ref{hypo:valRing} is
  true for $n = 1$, and then Lemma~\ref{lem:torsion_in_powers} says that this
  implies it is true for all $n ≥ 0$.
\end{rema}

We repeat Theorem~\ref{hypH} here as well for ease of reference.

\setcounter{dummy}{7}
\begin{hypoAlph}[{Theorem~\ref{hypH}}]\label{hypo:hypHhyp}
  Given a perfect field $k$, assume that for every integral $Y ∈ \Sch(k)$, every
  number $n ∈ \bN$ and every $ω ∈ \tor Ω^n(Y)$, there is a birational proper
  morphism $π: \widetilde{Y}→ Y$ such that the image of $ω$ in
  $Ω^n(\widetilde{Y})$ vanishes.
\end{hypoAlph}

\begin{propo}\label{prop:H-V}
  Let $k$ be perfect.  Then Hypothesis~\ref{hypo:valRing} and
  Hypothesis~\ref{hypo:hypHhyp} for $k$ are equivalent.  In particular,
  Hyothesis~\ref{hypo:hypHhyp} is true.
\end{propo}
\begin{proof}
  For the reader's convenience, the proof is subdivided into steps.

  \subsubsection*{Step 1: Proof \ref{hypo:valRing}$\,\Rightarrow\,$\ref{hypo:hypHhyp} in the affine case}

  Let $X = \Spec\, A ∈ \Sch(k)$ be integral and $ω ∈ Ω^n(X)$ an element which
  vanishes on a dense open, that is, the image of $ω$ in $Ω^n \bigl( \Frac(A)
  \bigr)$ is zero.  We wish to find a proper birational morphism $Y → X$ such
  that $ω|_Y = 0$.

  Hypothesis~\ref{hypo:valRing} implies then that the image of $ω$ in $Ω^n(R)$
  is zero for every valuation ring $R$ of $\Frac(A)$ which contains $A$.  As
  each $R$ is the union of its finitely generated sub-$A$-algebras, for each
  such $R$ there is a finitely generated sub-$A$-algebra, say $A_R$, for which
  $ω$ vanishes in $Ω^n(A_R)$.  The $E(A_R)$ then form an open cover of $\RZ(A)$
  and so since it is quasi-compact, there exists a finite subcover.  That is,
  there is a finite set $\{A_i\}_{i = 1}^m$ of finite generated sub-$A$-algebras
  of $\Frac(A)$ such that $ω$ is zero in each $Ω^n(A_i)$, and every valuation
  ring of $\Frac(A)$ which contains $A$, also contains one of the $A_i$'s.

  Zariski's Main Theorem in the form of Grothendieck,
  \cite[Théorème~8.12.6]{EGAIV3}, allows us, for each $i$, to choose a
  factorisation $\Spec\, A_i → Y_i → X$ as a dense open immersion followed by a
  proper morphism and to define $Y := Y_1 {⨯_X} \dots {⨯_X} Y_m$ (or better,
  define $Y$ to be the closure of the image of $\Spec\, \Frac(A)$ in this
  product) so that $Y → X$ is now a proper birational morphism.  Since $ω|_{A_i}
  = 0$ for each $i$, it suffices now to show that the set of open immersions $\{
  (\Spec\, A_i) ⨯_{Y_i} Y → Y \}_{i = 1}^m$ is an open cover of $Y$ to conclude
  that $ω|_Y = 0$.  But for every point $y ∈ Y$, there exists a valuation ring
  $R_y$ of $\Frac(A)$ such that $\Spec\, R_y → Y$ sends the closed point of
  $\Spec\, R_y$ to $y$, and since $R_y$ contains some $A_i$, there is a
  factorisation $\Spec\, R_y → \Spec\, A_i → Y$, and we see that $y ∈ (\Spec\,
  A_i) ⨯_{Y_i} Y$.

  \subsubsection*{Step 2: Proof \ref{hypo:valRing}$\,\Rightarrow\,$\ref{hypo:hypHhyp} in general}

  For the case of a general integral $X ∈ \Sch(k)$ we use the same trick.  Take
  an affine cover $\{U_i \}_{i = 1}^m$ of $X$.  We have just seen that there
  exist proper birational morphisms $V_i → U_i$ such that $ω|_{V_i} = 0$ for
  each $i$.  Zariski's Main Theorem in the form of Grothendieck gives
  compactifications $V_i → Y_i → X$.  We set $Y := Y_1 {⨯_X} \dots {⨯_X} Y_m$ so
  that $Y → X$ is proper birational, and the same argument as above shows that
  $\{ V_i {⨯_{Y_i}} Y → Y \}_{i = 1}^m$ is an open cover.  Since $ω|_{V_i} = 0$
  for each $i$, this implies that $ω|_Y = 0$.

  \subsubsection*{Step 3: Proof \ref{hypo:hypHhyp}$\,\Rightarrow\,$\ref{hypo:valRing}}

  Let $K$ be a finitely generated extension of $k$, let $R$ be a $k$-valuation
  ring of $K$, and let $ω$ be in the kernel of $Ω^n(R) → Ω^n(K)$.  There is some
  finitely generated sub-$k$-algebra $A$ of $R$ and $ω' ∈ Ω^n(A)$ such that
  $\Spec\, R → \Spec\, A$ is birational, and $ω'|_{R} = ω$.  Now by
  Hypothesis~\ref{hypo:hypHhyp} there is a proper birational morphism
  $\widetilde{Y} → \Spec\, A$ such that $ω'|_{\widetilde{Y}}$ is zero.  But by
  the valuative criterion for properness, there is a factorisation $\Spec\, R →
  \widetilde{Y} → \Spec\, A$, and so $ω = 0$.
\end{proof}

The following lemma says that if Hypothesis~\ref{hypo:valRing} is true for $n =
1$ then it is true for all $n ≥ 0$.

\begin{lemm}\label{lem:torsion_in_powers}
  Let $R$ be a (possibly non-discrete) valuation ring.  Let $M$ be a
  torsion-free $R$-module.  Then $\bigwedge^nM$ is torsion-free for all $n≥ 0$.
\end{lemm}
\begin{proof}
  We first consider the case where $M$ is finitely generated.  As a flat
  finitely generated module over a local ring, the module $M$ is even free of
  finite rank (see \cite[Proposition~3.G]{Mat} which is valid for non-noetherian
  rings).  Then $\bigwedge^nM$ is also free and hence torsion-free, $R$ being
  integral.

  For the general case let $M=\bigcup_{i∈ I}M_i$ with $M_i$ running through the
  system of all finitely generated submodules.  Note that the system is
  cofiltered.  Note that
  $$
  M^{\tensor n}=\bigcup_{i∈ I}M_i^{\tensor n}
  $$
  because tensor product commutes with direct limits.  By definition
  $$
  \bigwedge^nM=M^{\tensor n}/N
  $$
  where $N$ is generated by elementary tensors of the form $x_1\tensor
  x_2\tensor \dots\tensor x_n$ with $x_j=x_{j'}$ for some $j\neq j'$.  Hence
  $N=\bigcup_{i∈ I}N_i$ where $N_i$ is generated by elementary tensors as above
  with all $x_j∈ M_i$.  The direct limit of the sequences
  $$
  0 → N_i→ M^{\tensor n}_i→ \bigwedge^nM_i→ 0
  $$
  is
  $$
  0 → N → M^{\tensor n}→ \varinjlim_{i∈ I}\bigwedge^nM_i → 0
  $$
  and hence
  $$
  \bigwedge^nM=\varinjlim_{i∈ I}\bigwedge^{n}M_i.
  $$
  As a direct limit of torsion-free modules, it is also torsion-free.
\end{proof}

\subsection{Hyperplane section criterion}
\approvals{
  Annette & yes \\
  Shane & yes \\
  Stefan & yes
}
\label{app:A}

We give a criterion for testing the vanishing of torsion-forms.

\begin{lemm}[Hyperplane section criterion]\label{lemm:K2-1}
  Let $X$ be a quasi-projective variety, defined over an algebraically closed
  field $k$.  Let $H ⊂ X$ be any reduced, irreducible hyperplane that is not
  contained in the singular locus of $X$.  Then, the natural map $(\tor Ω¹_X)|_H
  → Ω¹_H$ is injective.

  In particular, if $U ⊆ X$ is open and $σ ∈ (\tor Ω¹_X)(U)$ is any torsion-form
  with induced form $σ_H ∈ Ω¹_H(U∩H)$, then $\supp σ ∩ H ⊆ \supp σ_H$.
\end{lemm}
\begin{proof}
  Consider the restriction $(\tor Ω¹_X)|_H$.  Since $H$ is not contained in the
  singular locus of $X$, this is a torsion-sheaf on $H$.  Recalling the exact
  sequence that defines torsion- and torsion-free forms,
  Sequence~\eqref{eq:torSeq} on page \pageref{eq:torSeq}, and using that
  $\check{Ω}¹_X$ is torsion-free and $\operatorname{Tor}_1^{\sO_X}(\sO_H,
  \check{Ω}¹_X)$ hence zero, observe that this sequence restricts to an exact
  sequence of sheaves on $H$,
  $$
  \xymatrix{ %
    0 \ar[r] & (\tor Ω¹_X)|_H \ar[r]^(.55){α} & Ω¹_X|_H \ar[r]^{β} &
    \check{Ω}¹_X|_H \ar[r] & 0.  %
  }
  $$
  In particular, $α$ injects the torsion-sheaf $(\tor Ω¹_X)|_H$ into the middle
  term of the second fundamental sequence for differentials,
  \cite[II.~Proposition~8.12]{Ha77},
  $$
  \xymatrix{ %
    & (\tor Ω¹_X)|_H \ar[d]_{α}^{\text{injection}} \\
    \factor{\sI_H}{\sI²_H} \ar[r]_(.55)a & Ω¹_X|_H \ar[r]_b & Ω¹_H \ar[r] & 0
  }
  $$
  We claim that the morphism $a$ is also injective.  To this end, recall that
  $\sI_H$ is locally principal.  In particular, $\factor{\sI_H}{\sI²_H}$ is a
  locally free sheaf of $\sO_H$-modules and that the morphism $a$ is injective
  at the generic point of $H$, where $X$ is smooth, \cite[Exercise~17.12]{E95}.
  It follows that $\ker a$ is a torsion-subsheaf of the torsion-free sheaf
  $\factor{\sI_H}{\sI²_H}$, hence zero.  The image $\operatorname{img} a$ is
  hence isomorphic to $\factor{\sI_H}{\sI²_H}$, and in particular torsion-free.
  As a consequence, note that the $\operatorname{img} α$, which is the image of
  a torsion-sheaf and hence itself torsion, intersects $\operatorname{img} a =
  \ker b$ trivially.  The composed morphism $b ◦ α$ is thus injective.  This
  proves the main assertion of Lemma~\ref{lemm:K2-1}.

  To prove the ``In particular …''-clause, let $U ⊆ X$ be any open set, $σ ∈
  (\tor Ω¹_X)(U)$ be any torsion-form and $x ∈ \supp σ$ be any closed point.  It
  follows from Nakayama's lemma that $x ∈ \supp(σ|_{H ∩ U}) ⊆ \supp \bigl( (\tor
  Ω¹_X)|_{H ∩ U} \bigr)$.  Since $H$ is not contained in the singular locus, the
  sheaf $Ω¹_X$ is locally free at the generic point of $H$, and $σ_H$ is thus a
  torsion-form on $H ∩ U$.  Since $b ◦ α$ is injective, its support contains $x$
  as claimed.  This finishes the proof of Lemma~\ref{lemm:K2-1}.
\end{proof}

This lemma and Flenner's Bertini-type theorems, \cite{MR0460317}, imply
following two theorems.

\begin{theo}[Existence of good hyperplanes through normal points]\label{theo:K2-2}
  Let $X$ be a quasi-projective variety of dimension $\dim X ≥ 3$, defined over
  an algebraically closed field $k$, and let $x ∈ X$ be a closed, normal point.
  Then, there exists a hyperplane section $H$ through $x$ such that $H$ is
  irreducible and reduced at $x$ and such that the following holds: if $U ⊆ X$
  is an open neighbourhood of $x$ and if $σ ∈ \tor Ω¹_X(U)$ is any torsion-form
  whose induced form $σ_H$ vanishes at $x$, then $σ$ vanishes at $x$.

  In particular, $Ω¹_X$ is torsion-free at $x$ if $Ω¹_H$ is torsion-free at $x$.
\end{theo}
\begin{proof}
  It follows from normality of $x ∈ X$ that the local ring $\sO_{X,x}$ satisfies
  Serre's condition $(R_1)$, has $\operatorname{depth} \sO_{X,x} ≥ 2$ and that
  it is analytically irreducible, \cite[Theorem~on page 352]{MR0024158}.  We can
  thus apply \cite[Korollar~3.6]{MR0460317} and find a hyperplane section $H$
  through $x$ that is irreducible and reduced at $x$.  Shrinking $X$ if need be,
  we can assume without loss of generality that $H$ is irreducible and reduced.
  Lemma~\ref{lemm:K2-1} then yields the claim.
\end{proof}

\begin{theo}[Existence of good hyperplanes in bpf linear systems]\label{theo:K2-1}
  Let $X$ be a quasi-projective variety of dimension $\dim X ≥ 3$, smooth in
  codimension one and defined over an algebraically closed field.  Let $\bH$ be
  a finite-dimensional, basepoint-free linear system.  Then, there exists a
  dense, open subset $\bH° ⊆ \bH$ such that any hyperplane section $H ⊂ X$ which
  corresponds to a closed point of $\bH°$ is irreducible, reduced, and satisfies
  the following additional property: if $U ⊆ X$ is open and $σ ∈ \tor Ω¹_X(U)$
  is any torsion-form with induced form $σ_H ∈ Ω¹_H(U∩H)$, then $\supp σ ∩ H ⊆
  \supp σ_H$.

  In particular, if $Ω¹_H$ is torsion-free, then $\supp \tor Ω¹_X$ is finite and
  disjoint from~$H$.
\end{theo}
\begin{proof}
  Recall from Flenner's version of Bertini's first theorem,
  \cite[Satz~5.2]{MR0460317}, that any general hyperplane $H$ is irreducible and
  reduced.  The main assertion of Theorem~\ref{theo:K2-1} thus follows from
  Lemma~\ref{lemm:K2-1}.

  If $Ω¹_H$ is torsion-free, the support of $\tor Ω¹_X$ necessarily avoids $H$.
  Since general hyperplanes can be made to intersect any positive-dimensional
  subvariety, we obtain the finiteness of $\supp \tor Ω¹_X$.
\end{proof}

%
%
\svnid{$Id: 0B.tex 243 2015-02-20 14:59:51Z kebekus $}

\section{Explicit computations}
\label{appendixCounterexample}
\subversionInfo
\approvals{
  Annette & yes \\
  Shane & yes\\
  Stefan & yes
}

Here we make some explicit calculations around
Example~\ref{exam:cdhDecsentFailure}.  We maintain notation and assumptions
introduced there.

\subsubsection{Differentials of $X$ and $\widetilde X$, and the pull-back map}

The modules of $k$-differentials are given as $Ω¹(X) = R· dx\oplus R · dy$ and
$Ω¹(\widetilde{X})=S · dx\oplus S · dz$.  In terms of these generators, the
pull-back map $dπ$ is given by
$$
dx\mapsto dx \quad\text{and}\quad dy \mapsto d(z^p+zx^n) = x^n · dz+nzx^{n-1} ·
dx.
$$

\subsubsection{The preimage of $Z$}
  
We have seen above that $Z = V(x)$ is a regular subvariety of $X$.  Its preimage
$\widetilde Z := π^{-1}(Z)$ is then given as
$$
\widetilde Z=\Spec\, \frac{k[x,y,z]}{(z^p+zx^n-y,x)} \isom \Spec\,
\frac{k[y,z]}{(z^p-y)} \isom \Spec\, k[z].
$$
In particular, we see that $\widetilde Z$ is likewise regular.

\subsubsection{Differentials of $Z$ and $\widetilde Z$, and the pull-back map}

The modules of $k$-differentials are given as $Ω¹(Z) =k[y] · dy$ and
$Ω¹(\widetilde Z) = k[z]· dz$.  In terms of these generators, the pull-back map
$d(π|_{\widetilde Z})$ is given by $dy \mapsto d(z^p)=0$.  The map
$d(π|_{\widetilde Z})$ is thus the zero map.

\subsubsection{Fibred products}

Finally, $\widetilde{X}⨯_X\widetilde{X}$ is the spectrum of the ring
\begin{multline*}
  \frac{k[x,y,z_1]}{(z_1^p+z_1x^n-y)} \tensor_{k[x,y]} \frac{k[x,y,z_2]}{(z_2^p+z_2x^n-y)} = \frac{k[x,y,z_1,z_2]}{(z_1^p+z_1x^n-y,z_2^p+z_2x^n-y)}\\
  \isom \frac{k[x,z_1,z_2]}{(z_2^p+z_2x^n-z_1^p-z_1x^n)} \isom
  \frac{k[x,z_1,u]}{(u^p+x^nu)},
\end{multline*}
with $u=z_2{-}z_1$.  Under this identification, the two projections
$\widetilde{X} {⨯_X} \widetilde{X} \rightrightarrows \widetilde{X}$ correspond
to the ring maps $z \mapsto z_1$ and $z \mapsto u + z_1$ ---notice that $y =
z_1^p {+} z_1 x^n = z_2^p {+} z_2 x^n$ in this ring.  The scheme $\widetilde{X}
{⨯_X} \widetilde{X}$ is regular outside of
$$
Z^{(2)} := V(x, u)=\Spec\, \frac{k[x,z_1,u]}{(u^p+x^nu,x, u)} \isom \Spec\, k[z_1] ⊂
\widetilde{X} ⨯_X \widetilde{X}.
$$

\subsubsection{Reg-differentials on $\widetilde{X}⨯_X\widetilde{X}$}

As $\widetilde{X}⨯_X\widetilde{X}$ is of dimension $≤ 2$, there is a resolution
of singularities, that is, a proper birational map $ρ: X^{(2)} →
\widetilde{X}⨯_X\widetilde{X}$, which is isomorphic outside $Z^{(2)}$.  Recall
from Remark~\ref{rem:defi:rhetc} that the obvious morphism $X^{(2)} \amalg
Z^{(2)} → \widetilde{X}⨯_X\widetilde{X}$ is a cover of the $\cdh$-topology.  We
have seen in Example~\ref{exam:ocdhs} that $Ω¹_{\rs}$ is a $\cdh$-sheaf.  We
obtain an injection
$$
Ω¹_{\rs}(\widetilde{X}⨯_X\widetilde{X}) → Ω¹_{\rs} \bigl(X^{(2)}\bigr) \oplus
Ω¹_{\rs}(Z^{(2)}).
$$
Recall that $X^{(2)}$ and $Z^{(2)}$ are both smooth.  This has two consequences.
First, we have observed in Remark~\vref{rema:firstRemarks:regularSame} that
reg-differentials and Kähler differentials agree,
$$
Ω¹_{\rs} \bigl(X^{(2)}\bigr)\oplus Ω¹_{\rs}(Z^{(2)}) \cong Ω¹
\bigl(X^{(2)}\bigr) \oplus Ω¹(Z^{(2)}).
$$
Secondly, the sheaves of Kähler-differentials are torsion-free, and inject into
rational differentials.  Summing up, we obtain an inclusion
$$
Ω¹_{\rs} \bigl( \widetilde{X} ⨯_X \widetilde{X} \bigr) → Ω¹ \left(
  \frac{k(x,z_1)[u]}{(u^p+x^nu)} \right) \oplus Ω¹ \bigl(k(z_1) \bigr).
$$

\end{appendix}

\approvals{
  Annette & yes\\
  Shane & yes\\
  Stefan & yes
}


\begin{thebibliography}{{Sta}14}

\bibitem[dJ96]{deJ96}
Aise~Johann de~Jong.
\newblock Smoothness, semi-stability and alterations.
\newblock {\em Inst. Hautes \'Etudes Sci. Publ. Math.}, (83):51--93, 1996.

\bibitem[Eis95]{E95}
David Eisenbud.
\newblock {\em Commutative algebra with a view toward algebraic geometry},
  volume 150 of {\em Graduate Texts in Mathematics}.
\newblock Springer-Verlag, New York, 1995.

\bibitem[Fer70]{Ferr70}
Aldo Ferrari.
\newblock Cohomology and holomorphic differential forms on complex analytic
  spaces.
\newblock {\em Ann. Scuola Norm. Sup. Pisa (3)}, 24:65--77, 1970.

\bibitem[Fle77]{MR0460317}
Hubert Flenner.
\newblock Die {S}\"atze von {B}ertini f\"ur lokale {R}inge.
\newblock {\em Math. Ann.}, 229(2):97--111, 1977.

\bibitem[Gei06]{Gei06}
Thomas Geisser.
\newblock Arithmetic cohomology over finite fields and special values of
  {$\zeta$}-functions.
\newblock {\em Duke Math. J.}, 133(1):27--57, 2006.

\bibitem[GL01]{GL}
Thomas~G. Goodwillie and Stephen Lichtenbaum.
\newblock A cohomological bound for the {$h$}-topology.
\newblock {\em Amer. J. Math.}, 123(3):425--443, 2001.

\bibitem[GR03]{GR03}
Ofer Gabber and Lorenzo Ramero.
\newblock {\em Almost ring theory}, volume 1800 of {\em Lecture Notes in
  Mathematics}.
\newblock Springer-Verlag, Berlin, 2003.

\bibitem[Gro66]{EGAIV3}
Alexandre Grothendieck.
\newblock Éléments de géométrie algébrique. {IV}. Étude locale des
  schémas et des morphismes de schémas. {T}roisiéme partie.
\newblock {\em Inst. Hautes Études Sci. Publ. Math.}, (28):255, 1966.
\newblock Revised in collaboration with Jean Dieudonné. Freely available on
  the Numdam web site at
  \url{http://www.numdam.org/numdam-bin/feuilleter?id=PMIHES_1966__28_}.

\bibitem[Gro67]{EGAIV4}
Alexandre Grothendieck.
\newblock \'{E}léments de géométrie algébrique. {IV}. \'{E}tude locale des
  schémas et des morphismes de schémas {IV}.
\newblock {\em Inst. Hautes Études Sci. Publ. Math.}, (32):361, 1967.
\newblock Revised in collaboration with Jean Dieudonné. Freely available on
  the Numdam web site at
  \url{http://www.numdam.org/numdam-bin/feuilleter?id=PMIHES_1967__32_}.

\bibitem[Har77]{Ha77}
Robin Hartshorne.
\newblock {\em Algebraic geometry}.
\newblock Springer-Verlag, New York, 1977.
\newblock Graduate Texts in Mathematics, No. 52.

\bibitem[HJ14]{HJ}
Annette Huber and Clemens Jörder.
\newblock Differential forms in the $h$-topology.
\newblock {\em Algebraic Geometry}, 4:449--478, 2014.
\newblock
  \href{http://dx.doi.org/10.14231/AG-2014-020}{DOI:10.14231/AG-2014-020}.
  Preprint \href{http://arxiv.org/abs/1305.7361}{arXiv:1305.7361}.

\bibitem[ILO14]{ILO}
Luc Illusie, Yves Laszlo, and Fabrice Orgogozo.
\newblock {\em Travaux de {G}abber sur l’uniformisation locale et la
  cohomologie étale des schémas quasi-excellents. (Séminaire à l’École
  polytechnique 2006–2008)}.
\newblock Number 363--364 in Ast{\'e}risque. Soci\'et\'e Math\'ematique de
  France, Paris, 2014.
\newblock Written in collaboration with F. Déglise, A. Moreau, V. Pilloni, M.
  Raynaud, J. Riou, B. Stroh, M. Temkin et W. Zheng. Freely available at
  \url{http://www.math.polytechnique.fr/~orgogozo/travaux_de_Gabber} and as
  \href{http://arxiv.org/abs/1207.3648}{arXiv:1207.3648}.

\bibitem[Keb13]{MR3084424}
Stefan Kebekus.
\newblock Pull-back morphisms for reflexive differential forms.
\newblock {\em Adv. Math.}, 245:78--112, 2013.
\newblock Preprint \href{http://arxiv.org/abs/1210.3255}{arXiv:1210.3255}.

\bibitem[Kel12]{Kel12}
Shane Kelly.
\newblock {\em Triangulated categories of motives in positive characteristic}.
\newblock PhD thesis, Universit{\'e} Paris 13, Australian National University,
  2012.
\newblock Preprint \href{http://arxiv.org/abs/1305.5349}{arXiv:1305.5349}.

\bibitem[KSY14]{KSY}
Bruno Kahn, Shuji Saito, and Takao Yamazaki.
\newblock Reciprocity sheaves, {I}.
\newblock 2014.
\newblock Preprint \href{http://arxiv.org/abs/1402.4201}{arXiv:1402.4201}.

\bibitem[Mat70]{Mat}
Hideyuki Matsumura.
\newblock {\em Commutative algebra}.
\newblock W. A. Benjamin, Inc., New York, 1970.

\bibitem[Mor12]{Mor12}
Fabien Morel.
\newblock {\em {$\Bbb A^1$}-algebraic topology over a field}, volume 2052 of
  {\em Lecture Notes in Mathematics}.
\newblock Springer, Heidelberg, 2012.

\bibitem[RG71]{RG}
Michel Raynaud and Laurent Gruson.
\newblock Crit\`eres de platitude et de projectivit\'e. {T}echniques de
  ``platification'' d'un module.
\newblock {\em Invent. Math.}, 13:1--89, 1971.

\bibitem[SGA72]{SGA41}
{\em Th\'eorie des topos et cohomologie \'etale des sch\'emas. {T}ome 1:
  {T}h\'eorie des topos}.
\newblock Lecture Notes in Mathematics, Vol. 269. Springer-Verlag, Berlin,
  1972.
\newblock S{\'e}minaire de G{\'e}om{\'e}trie Alg{\'e}brique du Bois-Marie
  1963--1964 (SGA 4), Dirig{\'e} par M. Artin, A. Grothendieck, et J. L.
  Verdier. Avec la collaboration de N. Bourbaki, P. Deligne et B. Saint-Donat.

\bibitem[{Sta}14]{stacks-project}
The {Stacks Project Authors}.
\newblock {\itshape Stacks Project}.
\newblock {\url{http://stacks.math.columbia.edu}}, 2014.

\bibitem[SV00]{SV00}
Andrei Suslin and Vladimir Voevodsky.
\newblock Bloch-{K}ato conjecture and motivic cohomology with finite
  coefficients.
\newblock In {\em The arithmetic and geometry of algebraic cycles ({B}anff,
  {AB}, 1998)}, volume 548 of {\em NATO Sci. Ser. C Math. Phys. Sci.}, pages
  117--189. Kluwer Acad. Publ., Dordrecht, 2000.

\bibitem[Voe96]{Voev96}
Vladimir Voevodsky.
\newblock Homology of schemes.
\newblock {\em Selecta Math. (N.S.)}, 2(1):111--153, 1996.

\bibitem[Zar48]{MR0024158}
Oscar Zariski.
\newblock Analytical irreducibility of normal varieties.
\newblock {\em Ann. of Math. (2)}, 49:352--361, 1948.

\bibitem[ZS75]{SZ}
Oscar Zariski and Pierre Samuel.
\newblock {\em Commutative algebra. {V}ol. {II}}.
\newblock Springer-Verlag, New York, 1975.
\newblock Reprint of the 1960 edition, Graduate Texts in Mathematics, Vol. 29.

\end{thebibliography}
\end{document}